\documentclass[10pt,reqno]{amsart}
\usepackage[hypertex]{hyperref}

\usepackage{xcolor}
\usepackage{amsmath,amsthm}
\usepackage{amscd}
\usepackage{amssymb}
\usepackage{euscript}
\usepackage[frame,ps,matrix,arrow,curve,rotate,all,2cell,tips]{xy}
\usepackage{epic,eepic}

\numberwithin{equation}{subsection}

\newcommand{\sqsp}{\renewcommand{\baselinestretch}{1.1}\tiny\normalsize}

\newtheorem{theorem}[subsection]{Theorem}
\newtheorem{lemma}[subsection]{Lemma}
\newtheorem{sublemma}[subsection]{Sublemma}
\newtheorem{proposition}[subsection]{Proposition}
\newtheorem{corollary}[subsection]{Corollary}

\theoremstyle{definition}
\newtheorem{definition}[subsection]{Definition}
\newtheorem{example}[subsection]{Example}
\newtheorem{remark}[subsection]{Remark}

\newcommand{\frakC}{\mathfrak{C}}

\newcommand{\sfO}{\mathsf{O}}
\newcommand{\sfP}{\mathsf{P}}
\newcommand{\sfQ}{\mathsf{Q}}

\newcommand{\Se}{\mathsf{Se}}

\newcommand{\bone}{\mathbf{1}}
\newcommand{\bSe}{\mathbf{Se}}

\newcommand{\catd}{\mathcal{D}}
\newcommand{\cate}{\mathcal{E}}
\newcommand{\catm}{\mathcal{M}}
\newcommand{\catp}{\mathcal{P}}

\newcommand{\catrcf}{\mathcal{RCF}}

\newcommand{\ua}{\underline{a}}
\newcommand{\ub}{\underline{b}}
\newcommand{\uc}{\underline{c}}
\newcommand{\ud}{\underline{d}}

\newcommand{\dc}{\binom{\ud}{\uc}}
\newcommand{\dcprime}{\binom{\ud'}{\uc'}}

\newcommand{\boxv}{\boxtimes_v}
\newcommand{\boxh}{\boxtimes_h}

\newcommand{\ec}{\cate^\frakC}
\newcommand{\sigmaec}{\mathbf{\Sigma}^\frakC_\cate} % C-colored sigma bimodule over E
\newcommand{\prop}{\mathbf{PROP}} % category of PROPs
\newcommand{\prope}{\mathbf{PROP}^\frakC_\cate} % category of C-colored PROPs over E
\newcommand{\operade}{\mathbf{Operad}^\frakC_\cate}
\newcommand{\vprope}{\mathbf{vPROP}^\frakC_\cate} % category of C-coloed vPROPs in E
\newcommand{\hprope}{\mathbf{hPROP}^\frakC_\cate} % category of C-coloed hPROPs in E
\newcommand{\mon}{\mathbf{Mon}} % monoids
\newcommand{\Ch}{\mathbf{Ch}} % chain complexes
\newcommand{\alg}{\mathbf{Alg}} % category of algebras

\DeclareMathOperator{\Hom}{Hom}
\DeclareMathOperator*{\colim}{colim}

\newcommand{\diagramit}[1]{\begin{equation}\SelectTips{cm}{10}\xymatrix{#1}\end{equation}}

\begin{document}

\title{On Homotopy Invariance for Algebras over Colored PROPs}
\author{Mark W. Johnson}
\email{mwj3@psu.edu}
\address{Department of Mathematics\\
         Penn State Altoona\\
         Altoona, PA 16601-3760}

\author{Donald Yau}
\email{dyau@math.ohio-state.edu}
\address{Department of Mathematics\\
    The Ohio State University at Newark\\
    1179 University Drive\\
    Newark, OH 43055}


\begin{abstract}
Over a monoidal model category, under some mild assumptions, we equip the categories of colored PROPs and their algebras with projective model category structures.  A Boardman-Vogt style homotopy invariance result about algebras over cofibrant colored PROPs is proved.  As an example, we define \emph{homotopy} topological conformal field theories and observe that such structures are homotopy invariant.
\end{abstract}


\date{\today}
\maketitle

\tableofcontents

\sqsp

%%%%%%%%%%%%%%%%%%%%%%
\section{Introduction}
%%%%%%%%%%%%%%%%%%%%%%

A PROP (short for PROduct and Permutation) is a very general algebraic machinery that can encode algebraic structures with multiple inputs and multiple outputs.  It was invented by Mac Lane \cite{maclane63} and was used prominently by Boardman and Vogt \cite{bv} in their seminal paper on homotopy invariant algebraic structures in algebraic topology.
By forgetting structures, every PROP $\sfP$ has an underlying operad $U\sfP$.  Operads are ``smaller" algebraic machines that encode algebraic structures with multiple inputs and one output.
PROPs are also widely used in string topology and mathematical physics.  Indeed, several types of topological field theories, such as topological conformal field theories (TCFT), are described by the so-called Segal PROP $\Se$ \cite{segalA,segalB,segalC}.  This is a topological PROP consisting of moduli spaces of Riemann surfaces with boundary holes.
Closely related is the PROP $\catrcf(g)$ \cite{chataur,cg} in string topology that is built from spaces of reduced metric Sullivan chord diagrams with genus $g$.

There is an important generalization of PROPs, called \emph{colored PROPs}, that can encode even more general types of algebraic structures, including morphisms and more general diagrams of algebras over a (colored) PROP.  The PROP $\catrcf(g)$ mentioned in the previous paragraph is actually a colored PROP.  Also, the Segal PROP $\Se$ has an obvious colored analogue in which the Riemann surfaces are allowed boundary holes with varying circumferences.  Multi-categories (a.k.a.\ colored operads) are to operads what colored PROPs are to PROPs.  In the simplest case, the set of colors $\frakC$ is the one-element set, and in this case we have a $1$-colored PROP.  For example, if $\sfP$ is a $1$-colored PROP, then diagrams of the form $A \to B$, consisting of two $\sfP$-algebras and a $\sfP$-algebra morphism between them, are encoded as algebras over a $2$-colored PROP $\sfP_{\bullet \to \bullet}$ \cite[Example 2.10]{fmy}.

The purpose of this paper is to study homotopy theory of colored PROPs and homotopy invariance properties of algebras over colored PROPs.  To ensure that our results are applicable in a variety of categories, we work over a symmetric monoidal closed category $\cate$ with a compatible model category structure \cite{quillen}.  To ``do homotopy theory" of colored PROPs over $\cate$ means that we lift the model category structure on $\cate$ to the category of colored PROPs.  Since colored PROPs are important mainly because of their algebras, under suitable conditions, we will also lift the model category structure on $\cate$ to the category of $\sfP$-algebras for a colored PROP $\sfP$ over $\cate$.

Several of the results here are generalizations of those in Berger and Moerdijk \cite{bm03,bm07}, in which the main focus are operads and colored operads.  The preprint \cite{harper} by Harper has results that are similar to those in \cite{bm03} with slightly different hypotheses and were obtained by different methods.  An earlier precedent was the work of Hinich \cite{hinich}, who obtained such model category structures in the setting of chain complexes and operads over them.  In \cite[p.373 (6)]{markl06a}, Markl stated the problem of developing a theory of homotopy invariant algebraic structures over PROPs in the setting of chain complexes.  In a recent preprint, Fresse \cite{fresse} independently studied homotopy theory of $1$-colored PROPs and homotopy invariance properties of algebras over $1$-colored PROPs in an arbitrary symmetric monoidal model category.

A description of some of our main results follows.  To study homotopy invariance properties of algebras, first we need a suitable model category structure on the category of colored PROPs.  The following result will be proved in Section \ref{modelstructure}, where the notion of strongly cofibrantly generated is defined in Definition \ref{modeldefs}.

\begin{theorem}[Model Category of Colored PROPs]
\label{thm:coloredpropmodel}
Let $\frakC$ be a non-empty set.  Suppose that $\cate$ is a strongly cofibrantly generated symmetric
monoidal model category with a symmetric monoidal fibrant replacement functor, and either:
\begin{enumerate}
\item
a cofibrant unit and a cocommutative interval, or
\item
functorial path data.
\end{enumerate}
Then the category $\prope$ of $\frakC$-colored PROPs over $\cate$ is a strongly cofibrantly generated model category with fibrations and weak equivalences defined entrywise in $\cate$.
\end{theorem}

We will define \emph{cocommutative interval} (Definition \ref{def:cocom}) and \emph{functorial path data} (Definition \ref{def:pathobj}) precisely in Section \ref{modelstructure}.  Both of these assumptions are used to construct path objects for fibrant $\frakC$-colored PROPs, which are needed to use the Lifting Lemma \ref{basiclift} to lift the model category structure.  For example, the category of simplicial sets has a cofibrant unit and a cocommutative interval \cite[Section 2]{bm07}.  On the other hand, the categories of chain complexes over a characteristic $0$ ring and of simplicial modules over a commutative ring admit functorial path data \cite[Section 5]{fresse}.  In particular, Theorem \ref{thm:coloredpropmodel} applies to each of these categories.  In fact, there is a modification of this projective structure on $\prope$, having more weak equivalences, which will be shown 
(Proposition \ref{prop:operadQequiv}) 
to be strongly Quillen equivalent to the projective structure on operads considered in \cite{bm07}.  As a consequence, one can view the homotopy theory of $\prope$ in the projective structure as a refinement of the homotopy theory of operads.

Next we describe our homotopy invariance result.  An algebraic structure is considered to be homotopy invariant if it can be transferred back and forth through a weak equivalence, at least under mild assumptions.  Historically, homotopy invariant structures in algebraic topology were first studied in the case of associative topological monoids.  Stasheff \cite{stasheff} found that a homotopy invariant substitute for an associative topological monoid is an $A_\infty$-space.  In \cite{bv} Boardman and Vogt vastly generalized this with the so-called $W$-construction.  Given a topological (colored) operad $\sfO$, the algebras over $W\sfO$ are homotopy invariants, where $W\sfO$ is a cofibrant resolution (i.e., cofibrant replacement) of $\sfO$.  Analogues of the results of Boardman and Vogt in the setting of chain complexes over a characteristic $0$ field were obtained by Markl \cite{markl04}.

One cannot expect that algebras over an arbitrary colored PROP $\sfP$ be homotopy invariants.  As suggested by the Boardman-Vogt $W$-construction, to obtain homotopy invariance of algebras, one should consider \emph{cofibrant} colored PROPs.  In fact, following Markl \cite[Principle 1]{markl04}, one can define \emph{homotopy algebras} as algebras over a cofibrant colored PROP.  The following result, which will be proved in Section \ref{homotopyinvariance}, says that homotopy algebras are homotopy invariants.

\begin{theorem}[Homotopy Invariance of Homotopy Algebras]
\label{thm:hinv}
Let $\frakC$ and $\cate$ be as in Theorem \ref{thm:coloredpropmodel}, and let $\sfP$ be a cofibrant $\frakC$-colored PROP in $\cate$.  Let $f = \{f_c \colon X_c \to Y_c\}_{c \in \frakC}$ be a collection of maps in $\cate$.  Then the following statements hold.
\begin{enumerate}
\item
Suppose that all of the $Y_{\underline{c}}$ are fibrant and that all of the $f_{\underline{c}} \colon X_{\underline{c}} \to Y_{\underline{c}}$ are acyclic cofibrations.  Then any $\sfP$-algebra structure on $X = \{X_c\}$ induces one on $Y = \{Y_c\}$ such that $f$ is a morphism of $\sfP$-algebras.
\item
Suppose that all of the $X_{\underline{c}}$ are cofibrant and that all of the $f_{\underline{c}} \colon X_{\underline{c}} \to Y_{\underline{c}}$ are acyclic fibrations.  Then any $\sfP$-algebra structure on $Y = \{Y_c\}$ induces one on $X = \{X_c\}$ such that $f$ is a morphism of $\sfP$-algebras.
\end{enumerate}
\end{theorem}

In the Theorem above, $X_{\underline{c}} = X_{c_1} \otimes \cdots \otimes X_{c_n}$ if $\uc = (c_1, \ldots , c_n)$ with each $c_i \in \frakC$, and similarly for $f_{\uc}$.  In the hypotheses of the two assertions, $\uc$ runs through all of the finite non-empty sequences of elements in $\frakC$.  Theorem \ref{thm:hinv} may be considered as a colored PROP, model category theoretic version of the Boardman-Vogt philosophy of homotopy invariant structures \cite{bv}.  It should be compared with \cite[(M1)-(M3), (M1')]{markl04} and \cite[Theorem 15 and Corollary 16]{markl02}.  Moreover, under suitable additional assumptions (see Theorem \ref{thm:algebralift}), there is a model category structure on the category of $\sfP$-algebras with entrywise fibrations and weak equivalences.  With this model category structure, we can therefore strengthen the conclusions in Theorem \ref{thm:hinv}: $f \colon X \to Y$ is a weak equivalence of $\sfP$-algebras in the first case and an acyclic fibration of $\sfP$-algebras in the second case.

As an illustration of Theorems \ref{thm:coloredpropmodel} and \ref{thm:hinv}, consider the category $\Ch(R)$ of non-negatively graded chain complexes over a field $R$ of characteristic $0$ with its usual model category structure \cite{ds,hovey,quillen}.  In Section \ref{sec:htcft}, we will recall the details of the Segal PROP $\Se$ mentioned earlier.  An algebra over the chain Segal PROP $\bSe = C_*(\Se)$ is called a \emph{topological conformal field theory} (TCFT) \cite{segalA,segalB,segalC}.  Denote by $\bSe_\infty$ the functorial cofibrant replacement of $\bSe$ given by Quillen's small object argument, and call an algebra over $\bSe_\infty$ a \emph{homotopy topological conformal field theory} (HTCFT).  A TCFT is also an HTCFT via an extended action associated to the acyclic fibration $\bSe_\infty \to \bSe$.

The following result, which will be discussed in Section \ref{sec:htcft}, is a special case of Theorem \ref{thm:hinv}.  Keep in mind that all objects in $\Ch(R)$ are both fibrant and cofibrant, which simplifies the statement.

\begin{corollary}[Homotopy Invariance of HTCFTs]
\label{cor:HTCFT}
Let $f \colon A \to B$ be a map in $\Ch(R)$, where $R$ is a field of characteristic $0$.
\begin{enumerate}
\item
Suppose that $f$ is an injective quasi-isomorphism.  Then any HTCFT structure on $A$ induces one on $B$ such that $f$ becomes a morphism of HTCFTs.
\item
Suppose that $f$ is a surjective quasi-isomorphism.  Then any HTCFT structure on $B$ induces one on $A$ such that $f$ becomes a morphism of HTCFTs.
\end{enumerate}
\end{corollary}

In other words, HTCFT is a homotopy invariant analogue of TCFT.

Next we consider the category $\alg(\sfP)$ of algebras over a colored PROP $\sfP$.  Let $\ec$ be the product category $\prod_{c \in \frakC} \cate$.  Given a $\frakC$-colored PROP $\sfP$ in $\cate$, each $\sfP$-algebra $X$ has an underlying object in $\ec$.  Let $U \colon \alg(\sfP) \to \ec$ be the forgetful functor.

\begin{theorem}[Model Category of Algebras]
\label{thm:algebralift}
Let $\frakC$ and $\cate$ be as in Theorem \ref{thm:coloredpropmodel}, and let $\sfP$ be a cofibrant $\frakC$-colored PROP in $\cate$.  Suppose in addition that
\begin{enumerate}
\item
finite tensor products of fibrations with fibrant targets remain fibrations in $\cate$, and
\item
the forgetful functor $U$ admits a left adjoint.
\end{enumerate}
Then there is a lifted model category structure on $\alg(\sfP)$ with fibrations and weak equivalences defined entrywise in $\ec$ .
\end{theorem}

As B. Fresse explained to the authors, for a general (colored) PROP $\sfP$, the forgetful functor $U$ may not admit a left adjoint, which is why we need this assumption.  On the other hand, if the monoidal category $\cate$ is Cartesian (i.e., the monoidal product is the Cartesian product), then $U$ always admits a left adjoint.   For example, any category of presheaves of sets is Cartesian, which includes simplicial sets.  The assumption about finite tensor products preserving fibrations is used in a technical result (Lemma \ref{factorcompat}).  Many symmetric monoidal model categories $\cate$ have this property, including chain complexes over a characteristic $0$ ring and all Cartesian categories.

Below we describe the organization of the rest of this paper.

In Section \ref{sec:coloredprop}, we define colored PROPs using two monoidal products $\boxv$ and $\boxh$.  Colored PROPs can be defined as $\boxv$-monoidal $\boxh$-monoids.  In Section \ref{modelstructure}, we lift the model category structure from the base category $\cate$ to the category of colored PROPs over $\cate$  (Theorem \ref{props'}), which is Theorem \ref{thm:coloredpropmodel} above.  A key ingredient is a Lifting Lemma \ref{basiclift}.  Different forms of this Lifting Lemma have been employed by various authors, e.g., \cite{hinich,ss}, to lift model category structures.  In Section \ref{homotopyinvariance}, we consider homotopy invariant versions of algebras over a colored PROP $\sfP$ and prove Theorem \ref{thm:hinv} above.  Section \ref{sec:htcft} provides an illustration of the Homotopy Invariance Theorem \ref{thm:hinv} and of the necessity of the (colored) PROP setting.  We discuss the Segal PROP $\Se$, its chain level algebras (TCFTs), and observe the validity of Corollary \ref{cor:HTCFT}.  In Section \ref{sec:admissible}, we consider the category of algebras over a colored PROP $\sfP$ and prove Theorem \ref{thm:algebralift} above.  In Section \ref{sec:Qpairs}, we consider Quillen pairs on the categories of colored $\Sigma$-bimodules, colored PROPs, and algebras induced by various adjoint pairs.

In the last two sections, we present some further results about colored PROPs for future references.  In Section \ref{app:coloredPROPs}, it is shown that colored PROPs can also be regarded as $\boxh$-monoidal $\boxv$-monoids.  In Section \ref{app:operads}, we discuss the free-forgetful (Quillen) adjunction between colored operads and colored PROPs.  Using this adjunction we observe that, for a colored operad $\sfO$, the categories of $\sfO$-algebras and $\sfO_{prop}$-algebras are equivalent, where $\sfO_{prop}$ is the colored PROP generated by $\sfO$.

\subsection*{Acknowledgement}
The authors thank Benoit Fresse for his helpful comments about earlier drafts of this paper and the anonymous referee for his/her helpful suggestions.  Author Johnson was partially supported by U.S.-Israel Binational Science Foundation grant 2006039.

%%%%%%%%%%%%%%%%%%%%%%%
\section{Colored PROPs}
\label{sec:coloredprop}
%%%%%%%%%%%%%%%%%%%%%%%

Our setting throughout this paper assumes $\mathcal{E} = (\cate, \otimes, I)$ is an underlying symmetric monoidal model category that is strongly cofibrantly generated with a zero object $0$.  The reader is referred to \cite{ss} for the definition of a monoidal model category (which subsumes the closed symmetric monoidal condition) and the pushout-product axiom.  It is assumed that $\cate$ in particular has all small colimits and limits.  We will say more about strong cofibrant generation in the next section.

To facilitate our discussion of model category structure on colored PROPs, in this section we give a formal definition of colored PROPs in $\cate$.  We work with colored PROPs without units.  As in the classical case of operads, we will build our colored PROPs from a form of $\Sigma$-objects.  Analogous to the description of operads as monoids with respect to the circle product, one can think of colored PROPs as \emph{monoidal monoids} (Proposition \ref{coloredprop} and Definition \ref{def:prop}).  This description of colored PROPs involves \emph{two} monoidal products $\boxv$ \eqref{boxvdc} and $\boxh$ \eqref{boxhsigmaobject}, the first one for the vertical composition and the other for the horizontal composition.  Thus, colored PROPs are monoids in a category of monoids.

We begin with a discussion of colors and colored $\Sigma$-bimodules.

%%%%%%%%%%%%%%%%%%%%%%%%%%%%%%%%%%%%%%%
\subsection{Colored $\Sigma$-bimodules}
\label{subsec:sigmabimodules}

Let $\frakC$ be a non-empty set, whose elements are called \textbf{colors}.  Our PROPs have a base set of colors $\frakC$.  The simplest case is when $\frakC = \{*\}$, which gives $1$-colored PROPs.

Let $\mathcal{P}(\frakC)$ denote the category whose objects, called \textbf{profiles} or \textbf{$\frakC$-profiles}, are finite non-empty sequences of colors.  If $\ud = (d_1, \ldots , d_m) \in \catp(\frakC)$, then we write $|\ud| = m$.  Our convention is to use a normal alphabet, possibly with a subscript (e.g., $d_1$) to denote a color and to use an underlined alphabet (e.g., $\ud$) to denote an object in $\catp(\frakC)$.

Permutations $\sigma \in \Sigma_{|\ud|}$ act on such a profile $\ud$ from the left by permuting the $|\ud|$ colors.  Given two profiles $\uc = (c_1, \ldots , c_n)$ and $\ud = (d_1, \ldots , d_m)$, a \textbf{morphism} in $\mathcal{P}(\frakC)$, $\uc \to \ud$ is a permutation $\sigma$ such that $\sigma(\uc) = \ud$.  Such a morphism exists if and only if $\ud$ is in the orbit of $\uc$.  Of course, if such a morphism exists, then $|\uc| = |\ud|$.  Clearly, this defines a groupoid and the \textbf{orbit type} of a $\frakC$-profile $\uc$, or equivalently the set of objects of the same connected component of the groupoid,  is denoted by $[\uc]$.

To emphasize that the permutations act on the profiles from the left, we will also write $\catp(\frakC)$ as $\catp_l(\frakC)$.  If we let the permutations act on the profiles from the right instead, then we get an equivalent category $\catp_r(\frakC)=\catp_l(\frakC)^{op}$.

Given profiles as above, we define the concatenation of $\uc$ and $\ud$,
\begin{equation}
\label{eq:concat}
(\uc, \ud) = (c_1, \ldots , c_n, d_1, \ldots , d_m) \in \catp(\frakC).
\end{equation}

The category of \textbf{$\frakC$-colored $\Sigma$-bimodules} over $\cate$ is defined to be the diagram category $\cate^{\catp_l(\frakC) \times \catp_r(\frakC)}$.  In other words, a $\frakC$-colored $\Sigma$-bimodule is a functor $\sfP \colon \catp_l(\frakC) \times \catp_r(\frakC) \to \cate$, and a morphism of $\frakC$-colored $\Sigma$-bimodules is a natural transformation of such functors. Unpacking this definition, a $\frakC$-colored $\Sigma$-bimodule $\sfP$ consists of the following data:

\begin{enumerate}
\item
For any profiles $\ud \in \catp_l(\frakC)$ and $\uc \in \catp_r(\frakC)$, it has an object
\[
\sfP\binom{\underline{d}}{\underline{c}} = \sfP\binom{d_1, \ldots , d_m}{c_1 , \ldots , c_n} \in \cate.
\]
This object should be thought of as the space of operations with $|\uc| = n$ inputs and $|\ud| = m$ outputs.   The $n$ inputs have colors $c_1, \ldots , c_n$, and the $m$ outputs have colors $d_1, \ldots , d_m$.

\item
Given permutations $\sigma \in \Sigma_{|\ud|}$ and $\tau \in \Sigma_{|\uc|}$, there is a morphism
\begin{equation}
\label{eq:sigmatauaction}
(\sigma; \tau) \colon \sfP\binom{\ud}{\uc} \to \sfP\binom{\sigma\ud}{\uc \tau}
\end{equation}
in $\cate$ such that:
\begin{enumerate}
\item $(1;1)$ is the identity morphism,
\item $(\sigma' \sigma; \tau\tau') = (\sigma';\tau') \circ (\sigma;\tau)$, and
\item $(1;\tau) \circ (\sigma; 1) = (\sigma; \tau) = (\sigma; 1) \circ (1; \tau)$.
\end{enumerate}
\end{enumerate}

Assembling these morphisms, there are commuting left $\Sigma_m$ action (acting on the $d_j$) and right $\Sigma_n$ action (acting on the $c_i$) on the object
\begin{equation}
\label{eq:Edecomp}
\sfP(m,n)
= \colim \sfP\binom{d_1, \ldots , d_m}{c_1 , \ldots , c_n}
= \colim \sfP\binom{\underline{d}}{\underline{c}},
\end{equation}
where the colimit is taken over all $\ud$ and $\uc$ with $|\ud| = m$ and $|\uc| = n$.  The object $\sfP(m,n)$ is said to have \textbf{biarity $(m,n)$}, and $\sfP\binom{\underline{d}}{\underline{c}}$ is called a \textbf{component} of $\sfP(m,n)$.  If there is only one color (i.e., $\frakC = \{*\}$), then $\sfP(m,n)$ has only one component, which has a left $\Sigma_m$ action and a right $\Sigma_n$ action that commute with each other.  We will sometimes abuse notation and refer to the left or right action of $\sigma$ on
$\sfP\binom{\underline{d}}{\underline{c}}$ to describe the structure maps from it to
$\sfP\binom{\sigma\underline{d}}{\underline{c}}$ or $\sfP\binom{\underline{d}}{\underline{c}\sigma}$.

Likewise, one observes that a morphism $f \colon \sfP \to \sfQ$ of $\frakC$-colored $\Sigma$-bimodules consists of color-preserving morphisms
\[
\left\{\sfP\binom{\underline{d}}{\underline{c}} \xrightarrow{f} \sfQ\binom{\underline{d}}{\underline{c}} \colon (\ud; \uc) \in \catp_l(\frakC) \times \catp_r(\frakC) \right\}
\]
that respect the $\Sigma_m$-$\Sigma_n$ action.  In other words, each square
\[
\SelectTips{cm}{10}
\xymatrix{
\sfP\dbinom{\ud}{\uc} \ar[r]^-f  \ar[d]_-{(\sigma;\tau)} & \sfQ\dbinom{\ud}{\uc} \ar[d]^-{(\sigma;\tau)} \\
\sfP\dbinom{\sigma\ud}{\uc\tau}\ar[r]^-f & \sfQ\dbinom{\sigma\ud}{\uc\tau}
}
\]
is commutative.

To simplify the notation, from now on the category of $\frakC$-colored $\Sigma$-bimodules over $\cate$ is denoted by $\sigmaec$ rather than $\cate^{\catp_l(\frakC) \times \catp_r(\frakC)}$.

%%%%%%%%%%%%%%%%%%%%%%%%%%%%%%%%%%%%%%%%%%%%%%
\subsection{Colored PROPs as monoidal monoids}

Recall that an operad can be equivalently defined as a monoid in the category of $\Sigma$-objects \cite[Lemma 9]{may97}.  We now describe a colored PROP analogue of this conceptual description of operads as monoids.  Since a colored PROP has both a vertical composition and a compatible horizontal composition, it makes sense that a colored PROP is described by \emph{two} monoidal products instead of just one.  First we build the vertical composition; then we build the horizontal composition on top of it.

To build the vertical composition, first we want to decompose $\sigmaec$ into smaller pieces with each piece corresponding to a pair of orbit types of colors.  To do this we need to use smaller indexing categories than all of the $\frakC$-profiles.

\begin{definition}
\label{sigmab}
Let $\ub = (b_1, \ldots , b_k)$ be a $\frakC$-profile.  Define the category $\Sigma_{\ub}$ to be the maximal connected sub-groupoid of $\catp_l(\frakC)$ containing $\ub$.

To introduce notation, this
is the category whose objects are the $\frakC$-profiles $\tau\ub = \left(b_{\tau(1)}, \ldots , b_{\tau(k)}\right) \in \catp(\frakC)$ obtained from $\ub$ by permutations $\tau \in \Sigma_k$.  Given two (possibly equal) objects $\tau \ub$ and $\tau'\ub$ in $\Sigma_{\ub}$, a morphism $\tau'' \colon \tau \ub \to \tau' \ub$ is a permutation in
$\Sigma_k$ such that $\tau''\tau \ub = \tau' \ub$ as $\frakC$-profiles.
\end{definition}

Notice that when we write $\tau\ub$ as an object in $\Sigma_{\ub}$, the permutation $\tau$ is not necessarily unique when $\ub$ contains repeated colors.  Indeed, $\tau'\ub$ is the same object as
$\tau\ub$ if and only if they are equal as ordered sequences of colors.  It is easy to see that there is an isomorphism $\Sigma_{\ub} \cong \Sigma_{\tau\ub}$ of groupoids for any $\tau \in \Sigma_{|\ub|}$.

\begin{example}
In the one-colored case, i.e., $\frakC = \{*\}$, a $\frakC$-profile $\ub$ is uniquely determined by its length $|\ub| = k$.  In this case, there is precisely one object $\ub = (*, \ldots , *)$ ($k$ entries) in the category $\Sigma_{\ub}$, since $\ub$ is unchanged by any permutation in $\Sigma_k$.
In other words, in the one-colored case, $\Sigma_{\ub}$ is the permutation group $\Sigma_{|\ub|}$, regarded as a category with one object, whose classifying space is $B\Sigma_{|\ub|}$.  \qed
\end{example}

\begin{example}
In the other extreme, suppose that $\ub = \left(b_1, \ldots , b_k\right)$ consists of distinct colors, i.e., $b_i \not= b_j$ if $i \not= j$.  There are now $k!$ different permutations of $\ub$, one for each $\tau \in \Sigma_k$.  So there are $k!$ objects in $\Sigma_{\ub}$.  Given two objects $\tau\ub$ and $\tau'\ub$ in $\Sigma_{\ub}$,
there is a unique morphism $\tau'\tau^{-1} \colon \tau\ub \to \tau'\ub$.  Thus, the classifying space
of $\Sigma_{\ub}$ in this case is a model for $E\Sigma_{|\ub|}$ as in \cite[5.9]{dwyer}.
\qed
\end{example}

To decompose $\frakC$-colored $\Sigma$-bimodules, we actually need a pair of $\frakC$-profiles at a time.  So we introduce the following groupoid.

\begin{definition}
Given any pair of $\frakC$-profiles $\ud$ and $\uc$, define $\Sigma_{\ud; \uc} = \Sigma_{\ud} \times \Sigma_{\uc}^{op}$, where $\Sigma_{\ud}$ and $\Sigma_{\uc}$ are introduced in Definition \ref{sigmab}.
\end{definition}

Of course, one could equivalently describe $\Sigma_{\ud; \uc}$ as the maximal connected sub-groupoid of $\catp_l(\frakC) \times \catp_r(\frakC)$ containing the ordered pair $(\ud;\uc)$.  If $\ud = (d_1, \ldots , d_m)$ and $\uc = (c_1, \ldots , c_n)$, then we write the objects in $\Sigma_{\ud; \uc}$ as pairs
\[
\binom{\sigma\ud}{\uc\tau} = \binom{d_{\sigma(1)}, \ldots , d_{\sigma(m)}}{c_{\tau^{-1}(1)}, \ldots , c_{\tau^{-1}(n)}}
\]
for $\sigma \in \Sigma_m$ and $\tau \in \Sigma_n$.

Given any $\frakC$-profile $\ud$, we denote by $[\ud]$ the orbit type of $\ud$ under permutations in $\Sigma_{|\ud|}$.  The following result is the decomposition of colored $\Sigma$-bimodules that we have been referring to.

\begin{lemma}
\label{decomposition}
There is a canonical isomorphism
\begin{equation}
\label{eq:decomp}
\sigmaec \cong \prod_{[\ud], [\uc]} \cate^{\Sigma_{\ud; \uc}},
\quad \sfP \mapsto \left\{\sfP\binom{[\ud]}{[\uc]}\right\}
\end{equation}
of categories, in which the product runs over all the pairs of orbit types of $\frakC$-profiles.
\end{lemma}

\begin{proof}
First recall that a (small) groupoid is canonically isomorphic to the coproduct of its connected components $\catp_l(\frakC) \times \catp_r(\frakC) \approx \coprod \Sigma_{\ud; \uc}$.
Now notice that the universal properties imply the category of functors
$Fun(\coprod \Sigma_{\ud; \uc},\cate)$ is canonically isomorphic to the product category
$\prod Fun(\Sigma_{\ud; \uc},\cate)$.
\end{proof}

\begin{example}
If $\frakC = \{*\}$, then the decomposition \eqref{eq:decomp} becomes
\[
\sigmaec \cong \prod_{m, n \geq 1} \cate^{\Sigma_m \times \Sigma_n^{op}}.
\]
An object in the diagram category $\cate^{\Sigma_m \times \Sigma_n^{op}}$ is simply an object $\sfP(m,n)$ in $\cate$ with a left $\Sigma_m$-action and a right $\Sigma_n$-action that commute with each other.\qed
\end{example}

An important ingredient in building the vertical composition in a colored PROP is a functor
\[
\otimes_{\Sigma_{\ub}} \colon \cate^{\Sigma_{\ud;\ub}} \times \cate^{\Sigma_{\ub;\uc}} \to \cate^{\Sigma_{\ud;\uc}},
\]
which in the one-colored case reduces to the usual $\otimes_{\Sigma_k}$.  So let $X \in \cate^{\Sigma_{\ud;\ub}}$ and $Y \in \cate^{\Sigma_{\ub;\uc}}$.  First we specify what $X \otimes_{\Sigma_{\ub}} Y$ does to objects in $\Sigma_{\ud;\uc}$.

Fix an object $\binom{\sigma\ud}{\uc\mu} \in \Sigma_{\ud;\uc}$.  Consider the diagram $D = D(X,Y; \sigma\ud, \uc\mu) \colon \Sigma_{\ub} \to \cate$ defined as
\begin{equation}
\label{Dobj}
D(\tau\ub) = X\binom{\sigma\ud}{\ub\tau^{-1}} \otimes Y\binom{\tau\ub}{\uc\mu}
\end{equation}
for each object $\tau\ub \in \Sigma_{\ub}$.  (Note that $\ub\tau^{-1} = \tau\ub$ as $\frakC$-profiles.)  The image of a morphism $\tau'' \in \Sigma_{\ub}(\tau\ub, \tau'\ub)$ under $D$ is the map
\begin{equation}
\label{tau''}
X\binom{\sigma\ud}{\ub\tau^{-1}} \otimes Y\binom{\tau\ub}{\uc\mu} \xrightarrow{\binom{1}{\tau''^{-1}} \otimes \binom{\tau''}{1}}
X\binom{\sigma\ud}{\ub\tau'^{-1}} \otimes Y\binom{\tau'\ub}{\uc\mu}.
\end{equation}
Now we define the object
\begin{equation}
\label{XtensoroverbY}
\left(X \otimes_{\Sigma_{\ub}} Y\right)\binom{\sigma\ud}{\uc\mu}
= \colim D(X,Y; \sigma\ud, \uc\mu) \in \cate,
\end{equation}
the colimit of the diagram $D = D(X,Y; \sigma\ud, \uc\mu)$.

Next we define the image under $X \otimes_{\Sigma_{\ub}} Y$ of a morphism in $\Sigma_{\ud;\uc}$.  So consider the morphism
\[
(\sigma''; \mu'') \in \Sigma_{\ud;\uc}\left(\binom{\sigma\ud}{\uc\mu}, \binom{\sigma'\ud}{\uc\mu'}\right).
\]
For each morphism $\tau'' \in \Sigma_{\ub}(\tau\ub, \tau'\ub)$, the square
\[
\SelectTips{cm}{10}
\xymatrix{
X\dbinom{\sigma\ud}{\ub\tau^{-1}} \otimes Y\dbinom{\tau\ub}{\uc\mu} \ar[rr]^-{\binom{\sigma''}{1} \otimes \binom{1}{\mu''}} \ar[d]_-{\binom{1}{\tau''^{-1}} \otimes \binom{\tau''}{1}} & & X\dbinom{\sigma'\ud}{\ub\tau^{-1}} \otimes Y\dbinom{\tau\ub}{\uc\mu'} \ar[d]^-{\binom{1}{\tau''^{-1}} \otimes \binom{\tau''}{1}} \\
X\dbinom{\sigma\ud}{\ub\tau'^{-1}} \otimes Y\dbinom{\tau'\ub}{\uc\mu} \ar[rr]^-{\binom{\sigma''}{1} \otimes \binom{1}{\mu''}} & & X\dbinom{\sigma'\ud}{\ub\tau'^{-1}} \otimes Y\dbinom{\tau'\ub}{\uc\mu'}
}
\]
is commutative because both composites are equal to $\binom{\sigma''}{\tau''^{-1}} \otimes \binom{\tau''}{\mu''}$.  These commutative squares give us a map
\[
\binom{\sigma''}{1} \otimes \binom{1}{\mu''} \colon D(X,Y; \sigma\ud, \uc\mu) \to D(X,Y; \sigma'\ud, \uc\mu')
\]
in the diagram category $\cate^{\Sigma_{\ub}}$.  Taking colimits we obtain a map
\begin{equation}
\label{XtensoroverbYmap}
\left(X \otimes_{\Sigma_{\ub}} Y\right)(\sigma''; \mu'') \colon \left(X \otimes_{\Sigma_{\ub}} Y\right)\binom{\sigma\ud}{\uc\mu} \to \left(X \otimes_{\Sigma_{\ub}} Y\right)\binom{\sigma'\ud}{\uc\mu'} \in \cate.
\end{equation}

The naturality of the constructions above is clear.  So we have the following result.

\begin{lemma}
\label{tensoroverb}
There is a functor
\begin{equation}
\label{eq:tensoroverb}
\otimes_{\Sigma_{\ub}} \colon \cate^{\Sigma_{\ud;\ub}} \times \cate^{\Sigma_{\ub;\uc}} \to \cate^{\Sigma_{\ud;\uc}}, \quad (X, Y) \mapsto \left(X \otimes_{\Sigma_{\ub}} Y\right)
\end{equation}
which restricts to the usual $\otimes_{\Sigma_{|\ub|}}$ in the $1$-colored case.
\end{lemma}

The functor \eqref{eq:tensoroverb} should be compared with the tensoring over a category construction in \cite[VII Section 2]{mm}, but the details of this expanded version are exploited below.

\begin{example}
If $\frakC = \{*\}$, then the functor \eqref{eq:tensoroverb} takes the form
\[
\otimes_{\Sigma_{\ub}} \colon \cate^{\Sigma_m \times \Sigma_k^{op}} \times \cate^{\Sigma_k \times \Sigma_n^{op}} \to \cate^{\Sigma_m \times \Sigma_n^{op}}
\]
if $|\ub| = k$.  An object $X \in \cate^{\Sigma_m \times \Sigma_k^{op}}$ is an object $X \in \cate$ with a left $\Sigma_m$-action and a right $\Sigma_k$-action that commute with each other.  Likewise, an object $Y \in \cate^{\Sigma_k \times \Sigma_n^{op}}$ is an object in $\cate$ equipped with commuting $\Sigma_k$-$\Sigma_n$ actions.  Since $\Sigma_{\ub} = \Sigma_k$ (as a category with one object), the only object in the diagram $D(X,Y; \sigma\ud, \uc\mu)$ above is $X \otimes Y \in \cate$.  The map \eqref{tau''} now takes the form
\[
X \otimes Y \xrightarrow{g^{-1} \otimes g} X \otimes Y
\]
for $g \in \Sigma_k$.  Therefore, we have
\[
X \otimes_{\Sigma_{\ub}} Y
= \colim_{g \in \Sigma_k}\left(X \otimes Y \xrightarrow{g^{-1} \otimes g} X \otimes Y\right)
= X \otimes_{\Sigma_k} Y.
\]
The maps \eqref{XtensoroverbYmap}
\[
(X \otimes_{\Sigma_k} Y)(\sigma; \mu) \colon X \otimes_{\Sigma_k} Y \to X \otimes_{\Sigma_k} Y
\]
for $(\sigma; \mu) \in \Sigma_m \times \Sigma_n^{op}$ give the $\Sigma_m$-$\Sigma_n$ actions on $X \otimes_{\Sigma_k} Y$, which are induced by those on $X$ and $Y$.\qed
\end{example}

Using the construction $\otimes_{\Sigma_{\ub}}$, now we want to describe colored $\Sigma$-bimodules equipped with a vertical composition.  So let $\sfP = \left\{\sfP\binom{[\ud]}{[\uc]}\right\}$ and $\sfQ = \left\{\sfQ\binom{[\ud]}{[\uc]}\right\}$ be $\frakC$-colored $\Sigma$-bimodules over $\cate$.  Recall the decomposition $\sigmaec \cong \prod \cate^{\Sigma_{\ud;\uc}}$ (Lemma \ref{decomposition}), in which $\sfP\binom{[\ud]}{[\uc]} \in \cate^{\Sigma_{\ud;\uc}}$.  We define a functor
\[
\boxv \colon \sigmaec \times \sigmaec \to \sigmaec
\]
by setting
\begin{equation}
\label{boxvdc}
\left(\sfP \boxv \sfQ\right)\binom{[\ud]}{[\uc]} = \coprod_{[\ub]} \sfP\binom{[\ud]}{[\ub]} \otimes_{\Sigma_{\ub}} \sfQ\binom{[\ub]}{[\uc]} \in \cate^{\Sigma_{\ud;\uc}},
\end{equation}
where the coproduct is taken over all the orbit types of $\frakC$-profiles.

\begin{lemma}
\label{boxvmonoidal}
The functor $\boxv$ gives $\sigmaec$ the structure of a monoidal category.
\end{lemma}

\begin{proof}
The unit of $\boxv$ is the object $\bone \in \sigmaec$ defined as
\[
\bone\dc = \begin{cases} 0 & \text{if $[\ud] \not= [\uc]$},\\
I & \text{if $[\ud] = [\uc]$},\end{cases}
\]
where $I$ is the unit of $\cate$.  The morphism $(\sigma;\tau) \colon \bone\dc \to \bone\binom{\sigma\ud}{\uc\tau}$ is the identity map of either $0$ or $I$, depending on whether $[\ud]$ and $[\uc]$ are equal or not.

The required associativity of $\boxv$ boils down to the associativity of the construction $\otimes_{\Sigma_{\ub}}$.  In other words, we need to show that the diagram
\begin{equation}
\label{asstensoroverb}
\SelectTips{cm}{10}
\xymatrix{
\cate^{\Sigma_{\ud;\ub}} \times \cate^{\Sigma_{\ub;\uc}} \times \cate^{\Sigma_{\uc;\ua}} \ar[rr]^-{\otimes_{\Sigma_{\ub}} \times Id} \ar[d]_-{Id \times \otimes_{\Sigma_{\uc}}} & &  \cate^{\Sigma_{\ud;\uc}} \times \cate^{\Sigma_{\uc;\ua}} \ar[d]^-{\otimes_{\Sigma_{\uc}}} \\
\cate^{\Sigma_{\ud;\ub}} \times \cate^{\Sigma_{\ub;\ua}} \ar[rr]_-{\otimes_{\ub}} & & \cate^{\Sigma_{\ud;\ua}}
}
\end{equation}
is commutative.  Suppose that $X \in \cate^{\Sigma_{\ud;\ub}}$, $Y \in \cate^{\Sigma_{\ub;\uc}}$, $Z \in \cate^{\Sigma_{\uc;\ua}}$, and $\binom{\sigma\ud}{\ua\nu} \in \Sigma_{\ud;\ua}$.  Note that $\otimes$ (being a left adjoint) commutes with colimits and is associative.  Therefore, either one of the two composites in \eqref{asstensoroverb}, when applied to $(X,Y,Z)$ and then to $\binom{\sigma\ud}{\ua\nu}$, gives the object
\begin{multline*}
\colim \left[ X\binom{\sigma\ud}{\ub\tau^{-1}} \otimes Y\binom{\tau\ub}{\uc\mu^{-1}} \otimes Z\binom{\mu\uc}{\ua\nu} \xrightarrow{\binom{1}{\tau''^{-1}} \otimes \binom{\tau''}{\mu''^{-1}} \otimes \binom{\mu''}{1}}\right.\\ \left.X\binom{\sigma\ud}{\ub\tau'^{-1}} \otimes Y\binom{\tau'\ub}{\uc\mu'^{-1}} \otimes Z\binom{\mu'\uc}{\ua\nu}\right]
\end{multline*}
in $\cate$.  This is the colimit of a diagram $D$ in $\cate$ indexed by the category $\Sigma_{\ub} \times \Sigma_{\uc}$.  For each pair of morphisms $\left(\tau'' \in \Sigma_{\ub}(\tau\ub;\tau'\ub); \mu'' \in \Sigma_{\uc}(\mu\uc;\mu'\uc)\right) \in \Sigma_{\ub} \times \Sigma_{\uc}$, the diagram $D$ has a morphism as indicated.  A similar observation shows that the two composites in \eqref{asstensoroverb} agree on morphisms in $\Sigma_{\ud;\ua}$ as well.
\end{proof}

We will not use the unit $\bone$ of $\boxv$ below and will consider $(\sigmaec,\boxv)$ as a monoidal category without unit.

\begin{definition}
\label{def:vprop}
Denote by $\vprope$ the category of monoids in the monoidal category $(\sigmaec, \boxv)$ (without unit), whose objects are called \textbf{vPROPs}, with $v$ standing for \emph{vertical}.
\end{definition}

Unwrapping the definitions of $\boxv$ and an associative map $\sfP \boxv \sfP \to \sfP$, we have the following description of a vPROP.

\begin{proposition}
\label{vprop}
A vPROP consists of precisely the following data:
\begin{enumerate}
\item
An object $\sfP \in \sigmaec$.
\item
A \textbf{vertical composition}
\[
\sfP\binom{[\ud]}{[\ub]} \otimes_{\Sigma_{\ub}} \sfP\binom{[\ub]}{[\uc]} \xrightarrow{\circ} \sfP\binom{[\ud]}{[\uc]}
\]
in $\cate^{\Sigma_{\ud;\uc}}$ that is associative in the obvious sense.
\end{enumerate}
Moreover, a morphism of vPROPs is precisely a morphism in $\sigmaec$ that preserves the vertical compositions.
\end{proposition}

By considering the definition of $\otimes_{\Sigma_{\ub}}$, one can unwrap the vertical composition one step further and describe it as an associative map
\begin{equation}
\label{verticalcomp}
\sfP\binom{\ud}{\ub} \otimes \sfP\binom{\ub}{\uc} \xrightarrow{\circ} \sfP\binom{\ud}{\uc}
\end{equation}
in $\cate$ for any $\frakC$-profiles $\ub$, $\uc$, and $\ud$ (not just representatives of orbit types).  It is equivariant, in the sense that the diagram
\begin{equation}
\label{verticalequivariance}
\SelectTips{cm}{10}
\xymatrix{
\sfP\dbinom{\ud}{\ub\tau^{-1}} \otimes \sfP\dbinom{\tau\ub}{\uc} \ar@{=}[r] & \sfP\dbinom{\ud}{\ub\tau^{-1}} \otimes \sfP\dbinom{\tau\ub}{\uc} \ar[d]^-{\circ} \\
\sfP\dbinom{\ud}{\ub} \otimes \sfP\dbinom{\ub}{\uc} \ar[u]^-{(1;\tau^{-1}) \otimes (\tau;1)} \ar[r]^-{\circ} \ar[d]_-{(\sigma;1) \otimes (1;\mu)} & \sfP\dbinom{\ud}{\uc} \ar[d]^-{(\sigma;\mu)} \\
\sfP\dbinom{\sigma\ud}{\ub} \otimes \sfP\dbinom{\ub}{\uc\mu} \ar[r]^-{\circ} & \sfP\dbinom{\sigma\ud}{\uc\mu}
}
\end{equation}
is commutative for $\sigma \in \Sigma_{|\ud|}$, $\mu \in \Sigma_{|\uc|}$, and $\tau \in \Sigma_{|\ub|}$.

Next we build the horizontal composition in a colored PROP.  To do this, we need to construct a functor
\[
\boxdot \colon \cate^{\Sigma_{\ud;\uc}} \times \cate^{\Sigma_{\ub;\ua}} \to \cate^{\Sigma_{(\ud,\ub);(\uc,\ua)}}.
\]
This functor is used to construct a monoidal product $\boxh$ on $\vprope$.  We then use $\boxh$ to describe PROPs as monoids in $(\vprope, \boxh)$.  Remembering that $\vprope$ is the category of $\boxv$-monoids in $\sigmaec$, this says that PROPs are \emph{$\boxv$-monoidal $\boxh$-monoids}, or \emph{monoidal monoids} for short.

The functor $\boxdot$ is constructed as an inclusion functor followed by a left Kan extension.  Indeed, there is a(n external product) functor
\[
\iota \colon \cate^{\Sigma_{\ud;\uc}} \times \cate^{\Sigma_{\ub;\ua}} \to \cate^{\Sigma_{\ud} \times \Sigma_{\ub} \times \Sigma_{\uc}^{op} \times \Sigma_{\ua}^{op}}, \quad (X,Y) \mapsto X \otimes Y
\]
that sends $(X,Y) \in \cate^{\Sigma_{\ud;\uc}} \times \cate^{\Sigma_{\ub;\ua}}$ to the diagram $X \otimes Y$ with
\begin{equation}
\label{XtensorY}
(X \otimes Y)\left(\sigma\ud; \mu\ub; \uc\tau^{-1}; \ua\nu^{-1}\right) = X\binom{\sigma\ud}{\uc\tau^{-1}} \otimes Y\binom{\mu\ub}{\ua\nu^{-1}},
\end{equation}
and similarly for maps in $\Sigma_{\ud} \times \Sigma_{\ub} \times \Sigma_{\uc}^{op} \times \Sigma_{\ua}^{op}$.  On the other hand, the subcategory inclusion
\[
\left(\Sigma_{\ud} \times \Sigma_{\ub}\right) \times \left(\Sigma_{\uc}^{op} \times \Sigma_{\ua}^{op}\right) \xrightarrow{i} \Sigma_{(\ud,\ub);(\uc,\ua)} = \Sigma_{(\ud;\ub)} \times \Sigma_{(\uc,\ua)}^{op}
\]
induces a functor on the diagram categories
\begin{equation}
\label{ei}
\cate^{i} \colon \cate^{\Sigma_{(\ud,\ub);(\uc,\ua)}} \to \cate^{\Sigma_{\ud} \times \Sigma_{\ub} \times \Sigma_{\uc}^{op} \times \Sigma_{\ua}^{op}}.
\end{equation}
This last functor has a left adjoint given by left Kan extension
\begin{equation}
\label{leftKanext}
K \colon \cate^{\Sigma_{\ud} \times \Sigma_{\ub} \times \Sigma_{\uc}^{op} \times \Sigma_{\ua}^{op}} \to \cate^{\Sigma_{(\ud,\ub);(\uc,\ua)}},
\end{equation}
which is left adjoint to the functor $\cate^i$ (\cite[pp.236-240]{maclane}).  Then we define the functor $\boxdot$ as the composite $\boxdot = K\iota$.

\begin{lemma}
\label{boxdotassociative}
The functor
\[
\boxdot = K\iota \colon \cate^{\Sigma_{\ud;\uc}} \times \cate^{\Sigma_{\ub;\ua}} \to \cate^{\Sigma_{(\ud,\ub);(\uc,\ua)}}
\]
is associative in the obvious sense.
\end{lemma}

\begin{proof}
The associativity of $\boxdot$ is a consequence of the associativity of $\otimes$ in $\cate$ \eqref{XtensorY} and the universal properties of left Kan extensions.
\end{proof}

Now we want to use $\boxdot$ to define a monoidal product $\boxh$ on $\vprope$.  Let $\sfP$ and $\sfQ$ be vPROPs.  First define the object $\sfP \boxh \sfQ \in \sigmaec \cong \prod \cate^{\Sigma_{\ud;\uc}}$ (Lemma \ref{decomposition}) by setting
\begin{equation}
\label{boxhsigmaobject}
(\sfP \boxh \sfQ)\binom{[\ud]}{[\uc]} = \coprod_{\substack{\ud = (\ud_1, \ud_2) \\ \uc = (\uc_1, \uc_2)}} \sfP\binom{[\ud_1]}{[\uc_1]} \boxdot \sfQ\binom{[\ud_2]}{[\uc_2]} \in \cate^{\Sigma_{\ud;\uc}}
\end{equation}
for any pair of orbit types $[\ud]$ and $[\uc]$.

\begin{lemma}
\label{boxhmonoidal}
The definition \eqref{boxhsigmaobject} gives a monoidal product
\[
\boxh \colon \vprope \times \vprope \to \vprope
\]
on the category $\vprope$.
\end{lemma}

\begin{proof}
In \eqref{boxhsigmaobject} we already defined $\sfP \boxh \sfQ$ as an object in $\sigmaec$.  To make it into a vPROP, we need an associative vertical composition (Proposition \ref{vprop})
\[
(\sfP \boxh \sfQ)\binom{[\ud]}{[\ub]} \otimes_{\Sigma_{\ub}} (\sfP \boxh \sfQ)\binom{[\ub]}{[\uc]} \xrightarrow{\circ} (\sfP \boxh \sfQ)\binom{[\ud]}{[\uc]} \in \cate^{\Sigma_{\ud;\uc}}.
\]
Since $\boxh$ is defined as a coproduct, we only need to define $\circ$ when restricted to a typical summand of the source:
\begin{multline*}
\left[\sfP\binom{[\ud_1]}{[\ub_1]} \boxdot \sfQ\binom{[\ud_2]}{[\ub_2]}\right] \otimes_{\Sigma_{\ub}} \left[\sfP\binom{[\ub'_1]}{[\uc_1]} \boxdot \sfQ\binom{[\ub'_2]}{[\uc_2]}\right]\\
\xrightarrow{\circ} \sfP\binom{[\ud_1]}{[\uc_1]} \boxdot \sfQ\binom{[\ud_2]}{[\uc_2]} \hookrightarrow (\sfP \boxh \sfQ)\binom{[\ud]}{[\uc]}.
\end{multline*}
This restriction of $\circ$ is defined as the $0$ map, unless $\ub_1 = \ub'_1$ (which implies $\ub_2 = \ub'_2$), in which case this $\circ$ is induced by those on $\sfP$ and $\sfQ$ (in the form \eqref{verticalcomp}), using the left Kan extension description of $\boxdot = K\iota$.  The associativity of $\circ$ follows from those on $\sfP$ and $\sfQ$, the associativity of $\otimes_{\Sigma_{\ub}}$, and the naturality of the construction $\boxdot$.  So $\sfP \boxh \sfQ$ is indeed a vPROP.  The associativity of $\boxh$ follows from that of $\boxdot$ (Lemma \ref{boxdotassociative}) and the definition ~\eqref{boxhsigmaobject}.
\end{proof}

\begin{definition}
\label{def:prop}
Denote by $\prope$ the category of monoids in the monoidal category (without unit) $(\vprope, \boxh)$, whose objects are called \textbf{$\frakC$-colored PROPs}.
\end{definition}

Unwrapping the definition of an associative map $\sfP \boxh \sfP \to \sfP$ of vPROPs and using Proposition \ref{vprop}, we have the following description of $\frakC$-colored PROPs.

\begin{proposition}
\label{coloredprop}
A $\frakC$-colored PROP consists of exactly the following data:
\begin{enumerate}
\item
An object $\sfP \in \sigmaec$.
\item
An associative \textbf{vertical composition}
\[
\sfP\binom{[\ud]}{[\ub]} \otimes_{\Sigma_{\ub}} \sfP\binom{[\ub]}{[\uc]} \xrightarrow{\circ} \sfP\binom{[\ud]}{[\uc]}
\]
in $\cate^{\Sigma_{\ud;\uc}}$.
\item
An associative \textbf{horizontal composition}
\begin{equation}
\label{horizontalcomposition}
\sfP\binom{[\ud_1]}{[\uc_1]} \boxdot \sfP\binom{[\ud_2]}{[\uc_2]} \xrightarrow{\otimes} \sfP\binom{[\ud_1,\ud_2]}{[\uc_1,\uc_2]}
\end{equation}
in $\cate^{\Sigma_{(\ud_1,\ud_2);(\uc_1,\uc_2)}}$.
\end{enumerate}
The assembled map
\[
\otimes \colon \sfP \boxh \sfP \to \sfP
\]
in $\sigmaec$ is required to be a map of vPROPs, a condition called the \textbf{interchange rule}.  Moreover, a map of $\frakC$-colored PROPs is exactly a map in $\sigmaec$ that preserves both the vertical and the horizontal compositions.
\end{proposition}

As in the case of the vertical composition \eqref{verticalcomp}, we may go one step further in unwrapping the horizontal composition.  At the level of $\cate$, the horizontal composition consists of associative maps
\begin{equation}
\label{horizontalcomposition2}
\sfP\binom{\ud_1}{\uc_1} \otimes \sfP\binom{\ud_2}{\uc_2} \xrightarrow{\otimes} \sfP\binom{\ud_1,\ud_2}{\uc_1,\uc_2}
\end{equation}
for any $\frakC$-profiles $\ud_1$, $\ud_2$, $\uc_1$, and $\uc_2$.
These maps are bi-equivariant, in the sense that the square
\begin{equation}
\label{horizontalequivariance}
\SelectTips{cm}{10}
\xymatrix{
\sfP\dbinom{\ud_1}{\uc_1} \otimes \sfP\dbinom{\ud_2}{\uc_2} \ar[r]^-{\otimes} \ar[d]_-{(\sigma_1;\tau_1) \otimes (\sigma_2;\tau_2)} & \sfP\dbinom{\ud_1,\ud_2}{\uc_1,\uc_2} \ar[d]^-{(\sigma_1 \times \sigma_2; \tau_1 \times \tau_2)}\\
\sfP\dbinom{\sigma_1\ud_1}{\uc_1\tau_1} \otimes \sfP\dbinom{\sigma_2\ud_2}{\uc_2\tau_2} \ar[r]^-{\otimes} & \sfP\dbinom{\sigma_1\ud_1,\sigma_2\ud_2}{\uc_1\tau_1,\uc_2\tau_2}
}
\end{equation}
is commutative for all $\sigma_i \in \Sigma_{|\ud_i|}$ and $\tau_i \in \Sigma_{|\uc_i|}$.  In totally unwrapped form, the interchange rule says that the diagram
\begin{equation}
\label{interchangerule}
\SelectTips{cm}{10}
\xymatrix{
\left[\sfP\binom{\ud_1}{\ub_1} \otimes \sfP\binom{\ud_2}{\ub_2}\right] \otimes \left[\sfP\binom{\ub_1}{\uc_1} \otimes \sfP\binom{\ub_2}{\uc_2}\right] \ar[r]^-{\text{switch}}_-{\cong} \ar[dd]_-{(\otimes, \otimes)} &
\left[\sfP\binom{\ud_1}{\ub_1} \otimes \sfP\binom{\ub_1}{\uc_1}\right] \otimes \left[\sfP\binom{\ud_2}{\ub_2} \otimes \sfP\binom{\ub_2}{\uc_2}\right] \ar[d]^-{(\circ, \circ)}\\
& \sfP\binom{\ud_1}{\uc_1} \otimes \sfP\binom{\ud_2}{\uc_2} \ar[d]^-{\otimes}\\
\sfP\binom{\ud_1,\ud_2}{\ub_1,\ub_2} \otimes \sfP\binom{\ub_1,\ub_2}{\uc_1,\uc_2} \ar[r]^-{\circ} & \sfP\binom{\ud_1,\ud_2}{\uc_1,\uc_2}
}
\end{equation}
is commutative.

Note that the interchange rule \eqref{interchangerule} is symmetric with respect to the vertical and the horizontal compositions.  In fact, it is possible to describe $\frakC$-colored PROPs as monoidal monoids in the other order, i.e., as $\boxh$-monoidal $\boxv$-monoids.  We will prove this in Section \ref{app:coloredPROPs}.

The following Lemma, which will be needed in the following sections, follows easily from the construction (see \cite[2.5-2.8]{fmy} and comments above Theorem \ref{props'}) of (free) $\frakC$-colored PROPs and Proposition \ref{coloredprop}.  A slight modification of the discussion in \cite[4.3-4.5]{fresse} gives another proof.

\begin{lemma}
\label{colimits}
The category of $\frakC$-colored PROPs over $\cate$ has all small limits and colimits.  Filtered
colimits, reflexive coequalizers, and all limits are constructed entrywise as in the underlying
category $\sigmaec$ of $\frakC$-colored $\Sigma$-bimodules.
\end{lemma}

%%%%%%%%%%%%%%%%%%%%%%%%%%%%%%%%%%%%%%
\subsection{Colored endomorphism PROP}

Before we talk about $\sfP$-algebras, let us first spell out the colored endomorphism PROP construction through which a $\sfP$-algebra is defined. Given objects $X, Y$, and $Z$ in $\cate$, there is a natural map
\begin{equation}
\label{eq:eta}
\eta \colon Z^Y \otimes Y^X \to Z^X,
\end{equation}
which is adjoint to the composition of the maps $Z^Y \otimes Y^X \otimes X \to Z^Y \otimes Y \to Z$.  Here $Y^X \otimes X \to Y$ is the adjoint of the identity map on $Y^X$ and similarly for the right-most map.  The map $\eta$ is associative in the obvious sense.

\begin{definition}
\label{def:endoPROP}
A  $\frakC$-colored \emph{endomorphism PROP} $E_X$ is associated to a $\frakC$-graded object $X = \{X_c\}_{c\in \frakC}$ in $\cate$.  Given $m, n \geq 1$ and colors $c_1, \ldots , c_n$, $d_1, \ldots , d_m$, it has the component
\[
E_X\binom{\underline{d}}{\underline{c}}
= \left(X_{d_1} \otimes \cdots \otimes X_{d_m}\right)^{\left(X_{c_1} \otimes \cdots \otimes X_{c_n}\right)} = X_{\underline{d}}^{X_{\underline{c}}}.
\]
The $\Sigma_{\ud}$ and $\Sigma_{\uc}$ acts as expected, with $\Sigma_{\ud}$ permuting the $m$ factors
$X_{\underline{d}} = X_{d_1} \otimes \cdots \otimes X_{d_m}$ and $\Sigma_{\uc}$ permuting the $n$ factors in the exponent.  The horizontal composition in $E_X$ is given by the naturality of exponentiation.  The vertical composition is induced by the natural map $\eta \colon X_{\underline{d}}^{X_{\underline{c}}} \otimes
X_{\underline{c}}^{X_{\underline{a}}} \to X_{\underline{d}}^{X_{\underline{a}}}$ discussed above \eqref{eq:eta}.

Thinking of the case of pointed sets or modules, $E_X\binom{\underline{d}}{\underline{c}} = \Hom(X_{\underline{c}}, X_{\underline{d}})$, and the horizontal composition is simply tensoring of functions.  The vertical composition is composition of functions with matching colors.
\end{definition}

\begin{definition}
\label{def:coloredPROPalgebra}
For a $\frakC$-colored PROP $\sfP$, a \emph{$\sfP$-algebra} structure on $X$ is a morphism $\lambda \colon \sfP \to E_X$ of $\frakC$-colored PROPs.  In this case, we say that $X$ is a $\sfP$-algebra with structure map $\lambda$.
\end{definition}

As usual one can unpack this definition and, via adjunction, express the structure map as a collection of maps
\[
\lambda \colon \sfP\binom{\underline{d}}{\underline{c}} \otimes X_{\underline{c}} \to X_{\underline{d}}
\]
with $(\underline{d}; \underline{c}) \in \mathcal{P}_l(\frakC) \times \catp_r(\frakC)$ that are associative (with respect to both the horizontal and the vertical compositions) and bi-equivariant.

A \emph{morphism} $f \colon X \to Y$ of $\sfP$-algebras is a collection of maps $f = \{f_c \colon X_c \to Y_c\}_{c\in \frakC}$ such that the diagram
\begin{equation}
\label{algebramap}
\SelectTips{cm}{10}
\xymatrix{
\sfP\binom{\underline{d}}{\underline{c}} \otimes X_{\underline{c}} \ar[r]^-{\lambda_X} \ar[d]_{Id \otimes f_{\underline{c}}} & X_{\underline{d}} \ar[d]^{f_{\underline{d}}} \\
\sfP\binom{\underline{d}}{\underline{c}} \otimes Y_{\underline{c}} \ar[r]^-{\lambda_Y} & Y_{\underline{d}}
}
\end{equation}
commutes for all $m, n \geq 1$ and colors $c_1, \ldots , c_n$ and $d_1, \ldots , d_m$.  Here $f_{\underline{c}} = f_{c_1} \otimes \cdots \otimes f_{c_n}$.

%%%%%%%%%%%%%%%%%%%%%%%%%%%%%%%%%%%%%%%%%%
\section{Model structure on colored PROPs}
\label{modelstructure}
%%%%%%%%%%%%%%%%%%%%%%%%%%%%%%%%%%%%%%%%%%

The purpose of this section is to prove Theorem \ref{thm:coloredpropmodel}, i.e., to lift the model category structure on $\cate$ to the category $\prope$ of $\frakC$-colored PROPs over $\cate$ (Theorem \ref{props'}).  This is achieved by first lifting the model category structure on $\cate$ to the category $\sigmaec$ of $\frakC$-colored $\Sigma$-bimodules (Proposition \ref{sigmabimods}).  The resulting model category structure on $\sigmaec$ is then lifted to $\prope$ via a standard Lifting Lemma \ref{basiclift}.  A variation defining the model structure on $\sigmaec$ using only a subset of the orbits of profiles is also discussed briefly at the end of the section, since it could be useful for change of colors operations, as considered in \cite{cgmv}.

\begin{definition}
	\label{modeldefs}  Suppose $K$ is a set of morphisms in a model category $\cate$.  Then:
\begin{itemize}
\item	the {\it relative $K$-cell complexes} will denote
	those morphisms which can be written as a transfinite composition of cobase changes of
	morphisms in $K$.  (See \cite[10.5.8]{hir}.)
\item	we say {\it sources in $K$ are small} if for each $f:A \to B$ in $K$, there exists a cardinal
	$\kappa_A$ such that for every regular cardinal $\lambda \geq \kappa$ and
	every $\lambda$-sequence $X:\lambda \to \cate$ the natural map
\[\colim_{\beta < \lambda} \cate(A,X_\beta) \to \cate(A, \colim_{\beta < \lambda} X_\beta)
\]
	is a bijection.  (See \cite[10.4.1]{hir}.)
\item	$\cate$ is a {\it strongly cofibrantly generated} model category if there exists sets of maps
	$I$ and $J$ with sources in both sets small, and a map $p$ is a fibration (resp. acyclic
	fibration) if and only if $p$ has the right lifting property with respect to every
	morphism in $J$ (resp. $I$). (See \cite[11.1.1]{hir}.)
\item given a functor $R:\catd \to \cate$, an {\it $R$-fibration (resp. $R$-weak equivalence)}
	will denote a morphism $p$ in $\catd$ with $R(p)$ a fibration (resp. weak equivalence).  An
	{\it $R$-fibrant replacement} for an object $Y \in \catd$ will indicate an $R$-weak equivalence
	$g:Y \to Z$  with $Z$ an $R$-fibrant object.  The $R$-lifted structure on $\catd$
	will refer to these classes
	together with the $R$-cofibrations, defined as those maps with the left lifting property with respect
	to each $R$-fibration which is also an $R$-weak equivalence.
\item an {\it $R$-path object} (often just called a path object below) in $\cate$ for an object
	$X \in \cate$ will denote the intermediate object
	in a factorization of the diagonal map $X \to Path(X) \to X \times X$
	as an $R$-weak equivalence followed by an $R$-fibration.  (See \cite[7.3.2(3)]{hir}.)
\end{itemize}	
\end{definition}

\begin{remark}
	We emphasize by the adjective `strongly' that our sources are assumed to be small,
	rather than the weaker condition of assuming that our sources are small with respect to
	the class of relative $I$-cell complexes.  Important examples excluded by this stronger
	assumption include almost all topological examples, while simplicial examples,
	chain complexes and so forth satisfy this strengthened assumption.  This assumption
	greatly simplifies the exposition, and our later results on Quillen equivalences (see section
	\ref{sec:Qpairs}) help to justify this restriction.
		
	We will often use a stronger notion of path object, where the first morphism is assumed to
	be an acyclic cofibration, rather than just a weak equivalence.
\end{remark}

\begin{lemma}
    \label{basiclift}
    Suppose $\cate$ is a strongly cofibrantly generated model category and $R:\catd \to \cate$
    has a left adjoint $L$.  Then $\catd$ becomes a cofibrantly generated model category
    (and $(L,R)$ form a strong Quillen pair) under the $R$-lifted model structure provided:
\begin{itemize}
\item       $\catd$ has all small limits and colimits, while $R$ creates/preserves filtered colimits,
\item       there is a functorial $R$-fibrant replacement $Q$ in $\catd$, and
\item       for every $R$-fibrant object in $\catd$, there is a
		path object construction which is preserved by $R$.
\end{itemize}
\end{lemma}

For clarity of presentation, we separate the key portion of the proof, adapted from \cite[Lemma B.2]{sch}.

\begin{sublemma}
	\label{piece}
	Under the last two assumptions of Lemma \ref{basiclift},
	(retracts of) relative $L(J)$-cell complexes are $R$-weak equivalences.
\end{sublemma}	

\begin{proof}
	Suppose $j$ is a relative $L(J)$-cell complex.
	By R-fibrancy of $Q(X)$ and the adjunction argument for the RLP against $L(J)$,
	there is a lift $r$ in the following diagram.
\diagramit{
X \ar[r]^-{\eta} \ar[d]_{j} & {Q(X)} \ar[d] \\
Y \ar[r] \ar@{.>}[ur]^{r} & {*}
}
	This implies $R(rj)=R(r)R(j)$ is a weak equivalence in $\cate$ by the assumption of
	$\eta$ an $R$-weak equivalence.
	Since the weak equivalences in $\cate$ are precisely the maps whose image
	in the homotopy category are isomorphisms, it then suffices to verify there is a right inverse
	for $R(j) \approx R(Q(j))$ in the homotopy category of $\cate$.  Our candidate will be $R(r)$.
	
	Define a map $Y \to QY \times QY$ by taking $\eta$ on the first factor and $Q(j) r$ on the
	second.  Then for any path object $P(QY)$ for $QY$ one has the following diagram, and the
	indicated dotted lift
\diagramit{
{X} \ar[d]_{j} \ar[r] & {QX} \ar[r] & {P(QY)} \ar[d] \\
{Y} \ar[rr] \ar@{.>}[urr]^-{H} && {QY \times QY},
}
	where the map $QX \to QY \to P(QY)$ comes from functoriality of $Q$ and
	the construction of the path object $P(QY)$.
	Now apply $R$ to this diagram to yield
\diagramit{
{R(X)} \ar[d]_{R(j)} \ar[r] & {R(P(QY))} \ar[r]^{\approx} & {P(R(QY))} \ar[d] \\
{R(Y)} \ar[r] \ar@{.>}[ur]^-{R(H)}& {R(QY \times QY)} \ar[r]_-{\approx} & {R(QY) \times R(QY)},
}
	which exhibits a right homotopy between $R(Q(j)r)$ and $R(\eta)$. Since $R(QY)$
	is fibrant by assumption, it follows as usual for a
	fibrant target that $R(Q(j)r)$ and $R(\eta)$ are identified in the homotopy category.
	This suffices to imply $R(Q(j))R(r)$ is an isomorphism in the
	homotopy category of $\cate$, since
	 $R(\eta)$ is a weak equivalence by assumption.
\end{proof}

\begin{proof}[Proof of Lemma \ref{basiclift}]
	We can appeal to \cite[11.3.2]{hir}, with the assumption that $R$ preserves filtered colimits
	sufficient to verify that sources in $L(I)$ and $L(J)$ are small, since
	sources in $I$ and $J$ are assumed to be small.
	This implies that $L(I)$ and $L(J)$ permit the small object argument, satisfying condition
	11.3.2(1), while condition 11.3.2(2) is verified by the Sublemma.
\end{proof}

\begin{proposition}
    \label{sigmabimods}
    The category $\sigmaec$ of $\frakC$-colored $\Sigma$-bimodules  carries a
    projective strongly cofibrantly generated model structure,
    where fibrations and weak equivalences are defined entrywise
    (in $\mathcal{E}^{\mathcal{P}_l(\frakC) \times \mathcal{P}_r(\frakC)}$),
    while the entries of any cofibration are cofibrations
    in $\cate$.  In addition, the properties of being
    simplicial, or proper are inherited from $\cate$ in this structure.
\end{proposition}

\begin{proof}
Since this is a category of diagrams, Hirschhorn's \cite[11.6.1]{hir}
applies.  For the simplicial condition use \cite[11.7.3]{hir},
for the entries of cofibrations use \cite[11.6.3]{hir}, and for the proper condition
use \cite[13.1.14]{hir} (or the definitions and the claim on cofibrations).
\end{proof}

\begin{remark}
	Keeping in mind our decomposition of the category of bimodules as a product, we point
	out that each piece of the decomposition also carries a similar model structure by the
	same arguments.  As a consequence, a morphism is a cofibration precisely when each
	projection in the decomposition \eqref{eq:decomp} is a cofibration.
\end{remark}

When the classes of fibrations and weak equivalences in a model category are defined by considering all entries in an underlying structure, it has become common to call it a projective model structure.  We propose to refer to modified projective structures (relative to a subset of the entries) when we define fibrations as those maps having only a specified set of entries fibrations, and similarly for weak equivalences.  We will see later, in considering both changes of colors (Corollary \ref{colorQequivs}) and the comparison with colored operads (Proposition \ref{prop:operadQequiv}) that modified structures can be useful.  However, unless otherwise noted, we will focus on projective structures throughout the remainder of this article.  Let $\Sigma$ denote the groupoid $\catp_l(\frakC) \times \catp_r(\frakC)$ defined in Section \ref{subsec:sigmabimodules}.

\begin{corollary}
	\label{modproj}
	Given a subset of the components of $\Sigma$, we have a modified projective structure
	on $\sigmaec$ where fibrations and weak equivalences are defined only by considering entries
	in the chosen components.  Once again, any entry of a cofibration will remain a cofibration
	in $\cate$, and this structure also inherits the properties of being simplicial or proper from $\cate$.
\end{corollary}

\begin{proof}
	One approach follows from a slight modification of Hirschhorn's \cite[11.6.1]{hir}, where only
	the generating cofibrations from the chosen components are used to define $I$ and $J$.
	Hence, the class of cofibrations is contained in those for the projective structure, which
	suffices to imply the entries are cofibrations in $\cate$ as above.
	
	Alternatively, one exploits the decomposition \eqref{eq:decomp} and considers each
	factor as one of only two types.  One observes that the components not chosen have
	strongly cofibrantly generated model structures where all maps are acyclic fibrations,
	the cofibrations are precisely the isomorphisms, and the sets $I$ and $J$ both
	consist of only the identity map on the initial object, by \cite[11.3.1]{hir}.  Now apply the
	arguments in the previous proof to give the components chosen projective
	structures.  The required modified projective structure is then the product of these
	structures, as in \cite[11.1.10]{hir}.  The claim about cofibrations then follows from the
	previous remark.
\end{proof}

Let $U \colon \prope \to \sigmaec$ denote the underlying $\frakC$-colored $\Sigma$-bimodule functor.  In order to lift the above model category structure on $\sigmaec$ to $\prope$ using the Lifting Lemma \ref{basiclift}, we need path objects for $U$-fibrant $\frakC$-colored PROPs.  One way to obtain a path object construction for $U$-fibrant $\frakC$-colored PROPs is by using a \emph{cocommutative interval} in $\cate$, which we now discuss.

\begin{definition}
\label{def:cocom}
We say that $\cate$ admits a \emph{cocommutative interval} (called a \emph{cocommutative coalgebra interval} in \cite{bm07}) if the fold map $\nabla \colon I \sqcup I \to I$ can be factored as
\begin{equation}
\label{coassinterval}
I \sqcup I \xrightarrow{\alpha} J \xrightarrow{\beta} I,
\end{equation}
in which $\alpha$ is a cofibration and $\beta$ is a weak equivalence, $J = (J, \Delta)$ is a coassociative cocommutative comonoid, and $\alpha$ and $\beta$ are both maps of comonoids.
\end{definition}

For example, the categories of (pointed) simplicial sets and of symmetric spectra \cite{hss} both admit cocommutative intervals \cite[Section 2]{bm07}.

\begin{definition}
\label{def:pathobj}
We say that $\cate$ admits \emph{functorial path data} if there exist a symmetric monoidal functor $Path$ on $\cate$ and monoidal natural transformations $s \colon Id \to Path$, $d_0, d_1 \colon Path \to Id$ such that $X \xrightarrow{s} Path(X) \xrightarrow{(d_0,d_1)} X \times X$ is a path object for $X$ whenever $X$ is fibrant.
\end{definition}

This definition is adapted from Fresse \cite[Fact 5.3]{fresse}.  Fresse showed that functorial path data exists when $\cate$ is either the category of chain complexes over a characteristic $0$ ring or of simplicial modules, among others.  We would like to consider other examples by using the following technical result.

\begin{lemma}
	\label{onetwo}
	If $\cate$ admits a cocommutative interval and $I$ is cofibrant,
	then $\cate$ admits functorial path data.
\end{lemma}

\begin{proof}
	Let $J$ be a cocommutative interval \eqref{coassinterval} in $\cate$.  Then define
	$Path(X)=X^J$ with the required transformations coming from the diagram
\[X \cong X^I \xrightarrow{\beta^*} X^J \xrightarrow{\alpha^*} X^{I \sqcup I} \cong X \times X
\]
	(so $d_0$ and $d_1$ are the projections of $\alpha^*$).  Notice the transformations are
	monoidal since $\alpha$ and $\beta$ are assumed to be comonoidal, while $X^J$ is
	symmetric monoidal via
\[ X^J \otimes Y^J \to \left(X \otimes Y \right)^{J \otimes J} \xrightarrow{\Delta^*} (X \otimes Y)^J
\]	
	since $J$ is a coassociative, cocommutative comonoid.	
	
	When $X$ is fibrant, by the pushout-product axiom
	it follows from $\alpha$ a cofibration that $\alpha^*$ is a fibration.  Similarly, from $\beta$
	a weak equivalence between cofibrant objects it follows that $\beta^*$ is a weak
	equivalence.  Hence, $X^J$ serves as a path object when $X$ is fibrant.
\end{proof}

Next we want to lift the model category structure on $\sigmaec$ (Proposition \ref{sigmabimods}) to
$\prope$ using the Lifting Lemma \ref{basiclift}.
As mentioned above, the forgetful functor $U \colon \prope \to \sigmaec$ has a left adjoint, the
\emph{free $\frakC$-colored PROP} functor $F \colon \sigmaec \to \prope$.  The existence of this left adjoint under set-theoretic assumptions on $\cate$ can be determined using Freyd's adjoint functor theorem.  However, an explicit description of this functor in the $1$-colored case can be found in \cite[Proposition 57]{markl06}, \cite[1.2]{mv}, or \cite[Appendix A]{fresse}.
The functor $F$ is defined as a colimit over a certain groupoid of directed graphs.  A straightforward extension of this construction, in which the edges of the directed graphs are $\frakC$-colored, works for the $\frakC$-colored case \cite[2.5-2.8]{fmy}.

\begin{theorem}
\label{props'}
Let $\cate$ be a strongly cofibrantly generated symmetric monoidal model category with:
\begin{enumerate}
\item
a symmetric monoidal fibrant replacement functor, and
\item
functorial path data (e.g. with cofibrant unit and admitting a cocommutative interval by Lemma \ref{onetwo}).
\end{enumerate}
Then the category $\prope$ of $\frakC$-colored PROPs over $\cate$ is a strongly cofibrantly generated model category with fibrations and weak equivalences defined entrywise.
\end{theorem}

\begin{proof}
	We would like to apply Lemma \ref{basiclift} to the free-forgetful adjoint pair $(F,U)$ above.
	Recall the model structure on $\sigmaec$ comes from Proposition \ref{sigmabimods}
	and Lemma \ref{colimits} deals with the first condition.  For the second and third conditions,
	simply apply the assumed symmetric monoidal fibrant replacement functor and functorial
	path data entrywise.  The result remains
	a $\frakC$-colored PROP by the symmetric monoidal assumption, and gives a $U$-fibrant replacement
	or $U$-path object for $U$-fibrant objects by the entrywise definitions in $\sigmaec$.
\end{proof}

In fact, the same proof applies to our other categories of monoids as well, since the constructions are all symmetric monoidal and the left adjoint is the free monoid functor in both cases.

\begin{theorem}\label{thm:vhprops}
	Under the assumptions of Theorem \ref{props'}, the category of hPROPs (Definition \ref{hpropdef}) and the category
	of vPROPs are similarly strongly cofibrantly generated model categories with fibrations
	and weak equivalences defined entrywise.
\end{theorem}

\begin{remark}\label{rmk:subset}
	Rather than working with all entries, in Theorems \ref{props'} and \ref{thm:vhprops} one could instead choose a subset of the orbits of pairs of
	$\frakC$-profiles and define fibrations and
	weak equivalences by considering only those entries.  For example, one could use the
	`intersection' (see \cite[Definition 8.5 and Proposition 8.7]{ij})
	of the model structures given by lifting over each chosen evaluation functor and its left adjoint.
	Since there is such a modified projective structure on $\sigmaec$ (Corollary \ref{modproj}),
	and all other structures are lifted from there, such an approach is equally valid for $\prope$, vPROPs, and hPROPs.  This is particularly useful in applications dealing with change of colors,
	as discussed near the end of Section \ref{sec:Qpairs}, and a comparison between colored 	
	operads and colored PROPs (Proposition \ref{prop:operadQequiv}).
\end{remark}

%%%%%%%%%%%%%%%%%%%%%%%%%%%%%%%%%%%%%%%%%%%%%%%%%%%%
\section{Homotopy invariance of homotopy algebras}
\label{homotopyinvariance}
%%%%%%%%%%%%%%%%%%%%%%%%%%%%%%%%%%%%%%%%%%%%%%%%%%%%

The purpose of this section is to prove Theorem \ref{thm:hinv}, which says that, under some reasonable conditions, an algebra structure over a cofibrant colored PROP is a homotopy invariant.  The proof is adapted from \cite[Theorem 3.5]{bm03}. Throughout this short section, $f:X \to Y$ will denote a morphism in $\prod_{c\in \frakC} \cate$, and $\sfP \in \prope$ a cofibrant $\frakC$-colored PROP.

\begin{definition}
	 First, we define the $\frakC$-colored $\Sigma$-bimodule $E_{X,Y}$ as a mixed endomorphism construction, having components
\[
E_{X,Y}\binom{\underline{d}}{\underline{c}}
= (Y_{d_1} \otimes \cdots \otimes Y_{d_m})^{X_{c_1} \otimes \cdots \otimes X_{c_n}}
= Y_{\underline{d}}^{X_{\underline{c}}}
\]
for $(\underline{d}; \underline{c}) \in \mathcal{P}_l(\frakC) \times \mathcal{P}_r(\frakC)$.

Now define the relative endomorphism construction $E_f$ via the pullback square in $\sigmaec$
(recall pullbacks are defined entrywise):
\begin{equation}
\label{eq:pullbacksq}
\SelectTips{cm}{10}
\xymatrix{
E_f
\ar[r]^-{\overline{f}^*} \ar[d]_{\overline{f}_*} & E_X
\ar[d]^{f_*} \\
E_Y
\ar[r]^-{f^*} & E_{X,Y}
}
\end{equation}
The map $f^*$ is given on entries by
\[
f^* = Y_{\underline{d}}^{f_{\underline{c}}} \colon E_Y\binom{\underline{d}}{\underline{c}} = Y_{\underline{d}}^{Y_{\underline{c}}} \to Y_{\underline{d}}^{X_{\underline{c}}} = E_{X,Y}\binom{\underline{d}}{\underline{c}}.
\]
The map $f_*$ is given on entries by
\[
f_* = f_{\underline{d}}^{X_{\underline{c}}} \colon E_X\binom{\underline{d}}{\underline{c}} = X_{\underline{d}}^{X_{\underline{c}}} \to Y_{\underline{d}}^{X_{\underline{c}}} = E_{X,Y}\binom{\underline{d}}{\underline{c}}.
\]
\end{definition}

\begin{lemma}
	\label{splitsquare}
	The relative endomorphism construction $E_f \in \prope$, and both morphisms
	$\overline {f}^*$ and $\overline {f}_*$ lie in $\prope$.  Furthermore, $f:X \to Y$ is a
	morphism of $\sfP$-algebras if and only if the $\sfP$-algebra structures on $X$ and $Y$ both
	descend from the same morphism $\sfP \to E_f$ in $\prope$.
\end{lemma}

\begin{proof}
	To obtain the horizontal composition for $E_f$, consider the diagram
\[
\SelectTips{cm}{10}
\xymatrix{
E_f\binom{\underline{d}}{\underline{c}} \otimes E_f\binom{\underline{b}}{\underline{a}} \ar[r]^-{\overline{f}^* \otimes \overline{f}^*} \ar[d]_{\overline{f}_* \otimes \overline{f}_*} & E_X\binom{\underline{d}}{\underline{c}} \otimes E_X\binom{\underline{b}}{\underline{a}} \ar[rr]^-{\text{horizontal comp}} & & E_X\binom{\underline{d},\underline{b}}{\underline{c},\underline{a}} \ar[d]^-{f_*}\\
E_Y\binom{\underline{d}}{\underline{c}} \otimes E_Y\binom{\underline{b}}{\underline{a}} \ar[rr]^-{\text{horizontal comp}} & & E_Y\binom{\underline{d},\underline{b}}{\underline{c},\underline{a}} \ar[r]^-{f^*} & E_{X,Y}\binom{\underline{d},\underline{b}}{\underline{c},\underline{a}}.
}
\]
This diagram is commutative, so by the universal property of pullbacks there is a unique induced map
\[
\otimes \colon E_f\binom{\underline{d}}{\underline{c}} \otimes E_f\binom{\underline{b}}{\underline{a}} \to E_f\binom{\underline{d},\underline{b}}{\underline{c},\underline{a}},
\]
which is the horizontal composition in $E_f$.  The vertical composition in $E_f$ is defined similarly using the vertical compositions in $E_X$ and $E_Y$ and the universal property of pullbacks.  As a consequence, the morphisms $\overline {f}^*$ and $\overline {f}_*$ also lie in $\prope$ by inspection.

In fact, $E_f$ enjoys a stronger than usual universal property, saying it is as close to a pullback in $\prope$ as possible, given that $E_{X,Y}$ is only in $\sigmaec$.  In more detail, suppose
$\theta:\sfP \to E_f$ is a morphism of bimodules equivalent by the universal property of pullbacks to a
pair of maps in $\sigmaec$, $\theta_X:\sfP \to E_X$ and $\theta_Y:\sfP \to E_Y$ which agree when composed into $E_{X,Y}$.
The strengthened pullback condition is that $\theta$ underlies a morphism in $\prope$ precisely when
both $\theta_X$ and $\theta_Y$ underly morphisms in $\prope$.  One direction of this implication follows from $\overline {f}^*$ and $\overline {f}_*$ morphisms in $\prope$, so now suppose both $\theta_X$ and $\theta_Y$ are morphisms in $\prope$ and combine commutativity of the diagram above with the commutative diagram for $\theta_X$ (and similarly for $\theta_Y$)
\diagramit{
\sfP\binom{\underline{d}}{\underline{c}} \otimes \sfP\binom{\underline{b}}{\underline{a}}  \ar[d] \ar[r]
  & E_X\binom{\underline{d}}{\underline{c}} \otimes E_X\binom{\underline{b}}{\underline{a}} \ar[d] \\
\sfP\binom{\underline{d},\underline{b}}{\underline{c},\underline{a}} \ar[r] &
	E_X\binom{\underline{d},\underline{b}}{\underline{c},\underline{a}}    .
}
Uniqueness of the induced map $\sfP\binom{\underline{d}}{\underline{c}} \otimes \sfP\binom{\underline{b}}{\underline{a}}  \to E_f\binom{\underline{d},\underline{b}}{\underline{c},\underline{a}}$ then suffices to show $\theta$ is compatible with horizontal composition, and a similar argument verifies compatibility with vertical composition.

The second claim follows from the strengthened universal property above, since an inspection of the definition shows $f:X \to Y$ is a morphism of $\sfP$-algebras precisely when
\[
\SelectTips{cm}{10}
\xymatrix{
\sfP
\ar[r] \ar[d] & E_X
\ar[d]^{f_*} \\
E_Y
\ar[r]^-{f^*} & E_{X,Y}
}
\]
commutes in $\sigmaec$.
\end{proof}

\begin{proposition}
	\label{htpyalg}
	Suppose the morphism $\overline {f}^*$ (respectively, $\overline {f}_*$)
	has the right lifting property with respect to the initial morphism into $\sfP$ (in $\prope$).	 Then any $\sfP$-algebra structure on $X$ (respectively, on $Y$) extends to make $f$ a
	morphism of $\sfP$-algebras.
\end{proposition}

\begin{proof}
	If $\overline {f}^*$ has the right lifting property with respect to the initial morphism, then
	any morphism $\eta:\sfP \to E_X$ lifts to a morphism $\overline \eta:\sfP \to E_f$.  (In fact, this lift
	is unique up to homotopy over $E_X$.)  Now apply Lemma \ref{splitsquare}.
\end{proof}

\begin{remark}
In fact, given a morphism $f:X \to Y$, one could use this lifting property approach to define the class of PROPs for which algebra structures always transfer along $f$.  Note that in Theorem \ref{thm:hinv}, the hypotheses are made so that the maps $Y_{\underline{d}}^{f_{\underline{c}}}$ and $f_{\underline{d}}^{X_{\underline{c}}}$ are acyclic fibrations, so it suffices to require $\sfP$ to be cofibrant.  However, the same argument implies any cofibrant contractible PROP $\sfP$ has the property that algebra structures extend for any $f$ with $f^*$ (or $f_*$) fibrations.
\end{remark}

 \begin{proof}[Proof of Theorem \ref{thm:hinv}]
Under the hypotheses of (1) and an equivalent form of the pushout-product axiom, the map $f^* = Y_{\underline{d}}^{f_{\underline{c}}}$ (obtained by exponentiating a fibrant object to an acyclic cofibration) in \eqref{eq:pullbacksq} is an acyclic fibration in $\cate$.  Since acyclic fibrations in $\cate$ are closed under pullback, and acyclic fibrations in $\prope$ are defined entrywise, this implies $\overline{f}^*$
 is an acyclic fibration in $\prope$.   Similarly, $\overline{f}_*$ is an acyclic fibration under the hypotheses of (2).   As $\sfP \in \prope$ is assumed to be cofibrant, the claim follows in either case from Proposition \ref{htpyalg}.
\end{proof}

%%%%%%%%%%%%%%%%%%%%%%%%%%%%%%%%%%%%%%%%%%%%%%%%%%%%%
\section{Homotopy topological conformal field theory}
\label{sec:htcft}
%%%%%%%%%%%%%%%%%%%%%%%%%%%%%%%%%%%%%%%%%%%%%%%%%%%%%

The purpose of this section is to provide an example to illustrate (i) the Homotopy Invariance Theorem \ref{thm:hinv} and (ii) the necessity of our (colored) PROP approach.  We will define a homotopy version of topological conformal field theory, and observe that it is, in fact, a homotopy invariant.

Our discussion will mostly focus on the $1$-colored \emph{Segal PROP} (\S \ref{subsec:segalprop}) and topological conformal field theory (\S \ref{subsec:tcft}) following \cite{segalA,segalB,segalC} for clarity.
Some other sources that discuss TCFT in a similar context are \cite{cv,getzler}.  The Segal PROP comes from considering moduli spaces of Riemann surfaces with boundary holes, and it is natural to consider varying (positive) circumferences of these boundary holes as the colors in this setting.   The vertical composition for the Segal PROP comes from holomorphically sewing along boundary holes, so the colored generalization would only allow sewing when the circumferences match.

Considering varying circumferences in the boundary holes is not unprecedented.  For example, in the setting of string topology, there is a combinatorially defined colored PROP $\catrcf(g)$ \cite{chataur,cg} that is built from spaces of reduced metric Sullivan chord diagrams with genus $g$.  Such a Sullivan chord diagram is a marked \emph{fat graph} (also known as ribbon graph) that represents a surface with genus $g$ that has a certain number of input and output circles in its boundary.  These boundary circles are allowed to have different circumferences and these form the set of colors for the colored PROP
$\catrcf(g)$.  However, for the remainder of this section, we will take $\frakC=\{ * \}$.

%%%%%%%%%%%%%%%%%%%%%%%%%%%%%%%%%%%
\subsection{Segal PROP}
\label{subsec:segalprop}

For integers $m, n \geq 1$, let $\Se(m,n)$ be the moduli space of (isomorphism classes of) complex Riemann surfaces whose boundaries consist of $m + n$ labeled holomorphic holes that are mutually non-overlapping.  In the literature, $\Se(m,n)$ is sometimes denoted by $\widehat{\catm}(m,n)$.  The holomorphic holes are actually bi-holomorphic maps from $m + n$ copies of the closed unit disk to the Riemann surface.  The first $m$ labeled holomorphic holes are called the \emph{outputs} and the last $n$ are called the \emph{inputs}.  Note that these Riemann surfaces $M$ can have arbitrary genera and are \emph{not} required to be connected.

One can visualize a Riemann surface $M \in \Se(m,n)$ as a pair of ``alien pants" in which there are $n$ legs (the inputs) and $m$ waists (the outputs).  See Figure \ref{alienpants}.  With this picture in mind, such a Riemann surface is also known as a \emph{worldsheet} in the physics literature.  In this interpretation, a worldsheet is an embedding of closed strings in space-time.  We think of such a Riemann surface $M$ as a machine that provides an operation with $n$ inputs and $m$ outputs.

The collection of moduli spaces $\{\Se(m,n) \colon m, n \geq 1\}$ forms a ($1$-colored) topological PROP $\Se$, called the \emph{Segal PROP}, also known as the \emph{Segal category}.  This Segal PROP $\Se$ is an honest PROP, in the sense that it is not generated by an operad, hence the need for additional machinery.

Using the characterization of PROPs from Proposition \ref{coloredprop}, its horizontal composition
\[
\Se(m_1,n_1) \times \Se(m_2,n_2) \xrightarrow{\otimes \,=\, \sqcup} \Se(m_1 + m_2,n_1 + n_2)
\]
is given by disjoint union $M_1 \sqcup M_2$.  In other words, put two pairs of alien pants side-by-side.  Its vertical composition
\[
\Se(m,n) \times \Se(n,k) \xrightarrow{\circ} \Se(m,k), \quad (M,N) \mapsto M \circ N
\]
is given by holomorphically sewing the $n$ output holes (the waists) of $N$ with the $n$ input holes (the legs) of $M$.  The $\Sigma_m$-$\Sigma_n$ action on $\Se(m,n)$ is given by permuting the labels of the $m$ output and the $n$ input holomorphic holes.

\begin{center}
\begin{figure}
\label{alienpants}
\setlength{\unitlength}{.6mm}
\begin{picture}(120,70)(-20,0)
% bottom 2 holes
\qbezier(20,15)(20,17)(25,17)
\qbezier(25,17)(30,17)(30,15)
\qbezier(20,15)(20,13)(25,13)
\qbezier(25,13)(30,13)(30,15)
\qbezier(30,15)(30,18)(35,18) % connecting edge
\qbezier(35,18)(40,18)(40,15) % connecting edge
\qbezier(40,15)(40,17)(45,17)
\qbezier(45,17)(50,17)(50,15)
\qbezier(40,15)(40,13)(45,13)
\qbezier(45,13)(50,13)(50,15)
% left edge
\qbezier(20,15)(20,22.5)(15,30)
\qbezier(15,30)(10,37.5)(10,45)
% right edge
\qbezier(50,15)(50,22.5)(55,30)
\qbezier(55,30)(60,37.5)(60,45)
% top 3 holes
\qbezier(10,45)(10,43)(15,43)
\qbezier(15,43)(20,43)(20,45)
\qbezier(10,45)(10,47)(15,47)
\qbezier(15,47)(20,47)(20,45)
\qbezier(20,45)(20,42)(25,42) % connecting edge
\qbezier(25,42)(30,42)(30,45) % connecting edge
\qbezier(30,45)(30,43)(35,43)
\qbezier(35,43)(40,43)(40,45)
\qbezier(30,45)(30,47)(35,47)
\qbezier(35,47)(40,47)(40,45)
\qbezier(40,45)(40,42)(45,42) % connecting edge
\qbezier(45,42)(50,42)(50,45) % connecting edge
\qbezier(50,45)(50,43)(55,43)
\qbezier(55,43)(60,43)(60,45)
\qbezier(50,45)(50,47)(55,47)
\qbezier(55,47)(60,47)(60,45)
% donut holes
\qbezier(21,31)(23,29)(25,29) % left
\qbezier(25,29)(27,29)(29,31)
\qbezier(23,30)(23,31)(25,31)
\qbezier(25,31)(27,31)(27,30)
\qbezier(41,31)(43,29)(45,29) % right
\qbezier(45,29)(47,29)(49,31)
\qbezier(43,30)(43,31)(45,31)
\qbezier(45,31)(47,31)(47,30)
% bottom 3 holes
\qbezier(70,15)(70,17)(75,17)
\qbezier(75,17)(80,17)(80,15)
\qbezier(70,15)(70,13)(75,13)
\qbezier(75,13)(80,13)(80,15)
\qbezier(80,15)(80,18)(85,18) % connecting edge
\qbezier(85,18)(90,18)(90,15) % connecting edge
\qbezier(90,15)(90,17)(95,17)
\qbezier(95,17)(100,17)(100,15)
\qbezier(90,15)(90,13)(95,13)
\qbezier(95,13)(100,13)(100,15)
\qbezier(100,15)(100,18)(105,18) % connecting edge
\qbezier(105,18)(110,18)(110,15) % connecting edge
\qbezier(110,15)(110,17)(115,17)
\qbezier(115,17)(120,17)(120,15)
\qbezier(110,15)(110,13)(115,13)
\qbezier(115,13)(120,13)(120,15)
% top hole
\qbezier(90,45)(90,47)(95,47)
\qbezier(95,47)(100,47)(100,45)
\qbezier(90,45)(90,43)(95,43)
\qbezier(95,43)(100,43)(100,45)
% left edge
\qbezier(70,15)(70,22.5)(80,30)
\qbezier(80,30)(90,37.5)(90,45)
\qbezier(120,15)(120,22.5)(110,30)
\qbezier(110,30)(100,37.5)(100,45)
% donut hole
\qbezier(91,31)(93,29)(95,29)
\qbezier(95,29)(97,29)(99,31)
\qbezier(93,30)(93,31)(95,31)
\qbezier(95,31)(97,31)(97,30)
% input and output
\put(65,65){\makebox(0,0){outputs}}
\put(65,-2){\makebox(0,0){inputs}}
% input labels
\put(25,8){\makebox(0,0){$4$}}
\put(45,8){\makebox(0,0){$1$}}
\put(75,8){\makebox(0,0){$5$}}
\put(95,8){\makebox(0,0){$2$}}
\put(115,8){\makebox(0,0){$3$}}
% output labels
\put(15,52){\makebox(0,0){$2$}}
\put(35,52){\makebox(0,0){$4$}}
\put(55,52){\makebox(0,0){$1$}}
\put(95,52){\makebox(0,0){$3$}}
\end{picture}
\caption{An element in $\Se(4,5)$ with two connected components.}
\end{figure}
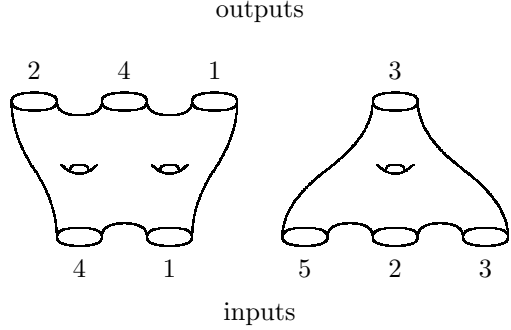
\end{center}

%%%%%%%%%%%%%%%%%%%%%%%%%%%%%%%%%%%%%%%%
\subsection{Topological conformal field theory}
\label{subsec:tcft}

We are interested in the chain level algebras of the Segal PROP $\Se$.  So we first pass to a suitable category of chain complexes.  Let $\Ch(R)$ be the category of non-negatively graded chain complexes of modules over a fixed commutative ring $R$, in applications a field of characteristic $0$.  Let $C_*$ be the singular chain functor with coefficients in $R$ from topological spaces to $\Ch(R)$.  Applying $C_*$ to the Segal PROP $\Se$, we obtain the \emph{chain Segal PROP} $\bSe = C_*(\Se)$, which is a ($1$-colored) PROP over $\Ch(R)$.

A \emph{topological conformal field theory} (TCFT) is defined as an $\bSe$-algebra.  It is also known as a \emph{string background} in the literature.  In other words, a TCFT has an underlying chain complex $A = \{A_n\}$ (the \emph{states space}) together with chain maps
\[
\lambda_{m,n} \colon \bSe(m,n) \otimes A^{\otimes n} \to A^{\otimes m}
\]
for $m, n \geq 1$ such that:
\begin{enumerate}
\item
The map $\lambda_{m,n}$ is $\Sigma_m$-$\Sigma_n$ equivariant.
\item
The following two diagrams are commutative, where $B(m,n) = \bSe(m,n) \otimes A^{\otimes n}$.
\[
\SelectTips{cm}{10}
\xymatrix{
\bSe(m_1,n_1) \otimes \bSe(m_2,n_2) \otimes A^{\otimes (n_1 + n_2)} \ar[r]^-{\text{switch}} \ar[d]_{\sqcup \otimes Id} & B(m_1,n_1) \otimes B(m_2,n_2) \ar[d]^{\lambda \otimes \lambda} \\
\bSe(m_1 + m_2,n_1 + n_2) \otimes A^{\otimes (n_1 + n_2)} \ar[r]^-{\lambda} & A^{\otimes m_1} \otimes A^{\otimes m_2},
}
\]
\begin{equation}
\label{sewingaxiom}
\SelectTips{cm}{10}
\xymatrix{
\bSe(m,n) \otimes \bSe(n,k) \otimes A^{\otimes k} \ar[r]^-{Id \otimes \lambda} \ar[d]_{\circ \otimes Id} & \bSe(m,n) \otimes A^{\otimes n} \ar[d]^\lambda \\
\bSe(m,k) \otimes A^{\otimes k} \ar[r]^-{\lambda} & A^{\otimes m}.
}
\end{equation}
\end{enumerate}
A \emph{morphism} of $TCFTs$ is a chain map $f \colon A \to B$ that is compatible with the structure maps $\lambda$ (see \eqref{algebramap}).

To simplify typography, we have abused notations and used $\sqcup$ and $\circ$ here to denote the images of the horizontal and vertical compositions in $\Se$ under the singular chain functor $C_*$. The commutativity of the second diagram \eqref{sewingaxiom} is usually called the \emph{sewing axiom}.  One can think of a TCFT as a chain level manifestation of topological string theory.

Now suppose that $f \colon A \to B$ is a quasi-isomorphism (i.e., a chain map that induces an isomorphism on homology).  Then from a homotopy point-of-view $A$ and $B$ should be considered ``the same."  If $A$ has a TCFT structure, then morally so should $B$, and vice versa, where $f$ becomes a morphism of TCFTs.  In other words:
\begin{center}
\emph{Topological conformal field theory, or a close variant of it, should be homotopy invariant}.
\end{center}
For the rest of this section, we will describe a suitable setting in which this is true.

%%%%%%%%%%%%%%%%%%%%%%%%%%%%%%%%%%%%%%%%%%%
\subsection{Homotopy TCFT}
We will fit TCFTs into the setting of the Homotopy Invariance Theorem \ref{thm:hinv} by enlarging the category to a homotopy version of TCFT.  To have a good notion of homotopy, let us first describe a model category structure on $\Ch(R)$.

In $\Ch(R)$, where $R$ is a field of characteristic $0$, a map $f \colon A \to B$ is a:
\begin{enumerate}
\item
\emph{weak equivalence} if and only if it is a quasi-isomorphism;
\item
\emph{fibration} if and only if $f_k \colon A_k \to B_k$ is a surjection for $k > 0$;
\item
\emph{cofibration} if and only if it is a degree-wise injection.
\end{enumerate}
Full details of the existence of this model category structure can be found in \cite{ds,quillen}.  Moreover, in \cite{hovey} it is shown that this model category is (strongly) cofibrantly generated.  Since the $0$ object in $\Ch(R)$ is the $0$ chain complex, every object $A \in \Ch(R)$ is both fibrant and cofibrant.  Note that by the $5$-Lemma, an acyclic fibration $f$ is surjective in dimension $0$ as well.  So the acyclic fibrations in $\Ch(R)$ are exactly the surjective quasi-isomorphisms.

As discussed in the Introduction, the model category $\Ch(R)$ satisfies the hypotheses of Theorem \ref{thm:coloredpropmodel}.  Therefore, the category $\prop(\Ch(R))$ of $\frakC$-colored PROPs over $\Ch(R)$ inherits a cofibrantly generated model category structure in which the fibrations and weak equivalences are defined entrywise.

Now consider the chain Segal PROP $\bSe$.  To use the Homotopy Invariance Theorem \ref{thm:hinv}, we need a cofibrant version of $\bSe$.  There is a cofibrant replacement functor \cite[10.5.16]{hir} in the model category $\prop(\Ch(R))$, provided by Quillen's small object argument.  Applying this cofibrant replacement functor to $\bSe$, we have an acyclic fibration
\begin{equation}
\label{cofibrantSegal}
r \colon \bSe_\infty \buildrel \sim \over \twoheadrightarrow \bSe
\end{equation}
in which $\bSe_\infty$ is a cofibrant PROP over $\Ch(R)$.  We call $\bSe_\infty$ \emph{the cofibrant chain Segal PROP}.  Since weak equivalences and fibrations in $\prop(\Ch(R))$ are defined entrywise, each chain map
\[
r(m,n) \colon \bSe_\infty(m,n) \buildrel \sim \over \twoheadrightarrow \bSe(m,n) = C_*(\Se(m,n))
\]
is a surjective quasi-isomorphism.

\begin{definition}
\label{htcft}
The category of \emph{homotopy topological conformal field theories} (HTCFTs), or \emph{homotopy string backgrounds}, is defined to be the category of $\bSe_\infty$-algebras.
\end{definition}

One can unpack this definition and interpret an HTCFT in terms of structure maps $\lambda_{m,n}$ as in section \ref{subsec:tcft}.  Notice that TCFTs are HTCFTs.  Indeed, a TCFT is given by a map $\lambda \colon \bSe \to E_A$ of PROPs over $\Ch(R)$, where $E_A$ is the endomorphism PROP of a chain complex $A$.  Composing with $r$ of \eqref{cofibrantSegal}, we have a map $\lambda r \colon \bSe_\infty \to E_A$ of PROPs, which gives an HTCFT structure on $A$.

We now obtain Corollary \ref{cor:HTCFT}, which is a restatement of Theorem \ref{thm:hinv} in the current setting: $\sfP = \bSe_\infty$ and $\cate = \Ch(R)$. Since TCFTs are, in particular, HTCFTs, we also have the following alternative version of Corollary \ref{cor:HTCFT}.  Again, keep in ming that all objects in $\Ch(R)$ being both fibrant and cofibrant simplifies the statement.

\begin{corollary}
\label{cor3:htcft}
Let $f \colon A \to B$ be a map in $\Ch(R)$, where $R$ is a field of characteristic $0$.
\begin{enumerate}
\item
Suppose that $f$ is an injective quasi-isomorphism.  If $A$ is a TCFT, then there exists an HTCFT structure on $B$ such that $f$ becomes a morphism of HTCFTs.
\item
Suppose that $f$ is a surjective quasi-isomorphism.  If $B$ is a TCFT, then there exists an HTCFT structure on $A$ such that $f$ becomes a morphism of HTCFTs.
\end{enumerate}
\end{corollary}

%%%%%%%%%%%%%%%%%%%%%%%%%%%%%%%%%%%%%
\section{Model structure on algebras}
\label{sec:admissible}
%%%%%%%%%%%%%%%%%%%%%%%%%%%%%%%%%%%%%

The purpose of this section is to prove Theorem \ref{thm:algebralift}, which says that, under some suitable assumptions on $\cate$, there is a lifted model category structure on the category of $\sfP$-algebras.

We begin with the following key technical result, keeping in mind that $\cate^\frakC$ has a projective model structure by \cite[11.6.1]{hir}.  This result can be considered as a colored PROP version of the discussion in \cite[\S 4.1.14]{rezk}, which focuses on operads.

\begin{lemma}
	\label{factorcompat}
	Suppose finite tensor products of fibrations with fibrant targets remain fibrations in $\cate$
	and that $\sfP$ is a cofibrant object of $\prope$.
	Then for any factorization of a morphism of $\sfP$-algebras
	$g:A \to C$ with (entrywise) fibrant target as
	$A \stackrel{f}{\to} B \stackrel{h}{\to} C$ with $f$ an acyclic cofibration in
	$\cate^\frakC$ and $h$ a fibration in $\cate^\frakC$, one may equip $B$ with a $\sfP$-algebra
	structure making both $f$ and $h$ morphisms of $\sfP$-algebras.
\end{lemma}	

\begin{proof}
	First, by Lemma \ref{splitsquare},
	$f$ is a morphism of $\sfP$-algebras precisely when the structure
	maps $\sfP \to E_A$ and $\sfP \to E_B$ both descend from a common morphism
	$\sfP \to E_f$ in $\prope$ and similarly for $h$.  If we form the following pullback
\diagramit{
E_1 \ar[d] \ar[r] & E_h \ar[d] \\
E_f \ar[r] & E_B
}	
	in $\prope$, we can extend it to form the commutative diagram
\diagramit{
E_1 \ar[d] \ar[r] & E_h \ar[d] \ar[r] & E_C \ar[d] \\
E_f \ar[d] \ar[r] & E_B \ar[d] \ar[r] & E_{B,C} \ar[d] \\
E_A \ar[r] & E_{A,B} \ar[r] & E_{A,C}
}	
in $\sigmaec$.
Notice $E_g$ is the pullback of the largest square here, so this induces a morphism in $\prope$,
$\theta:E_1 \to E_g$ by the strengthened universal property discussed in the proof of
Lemma \ref{splitsquare} and
the fact that the composites $E_1 \to E_C$ and $E_1 \to E_A$ are morphisms in $\prope$ by construction.  We will show $\theta$ is an acyclic fibration in $\prope$, hence that there must be a lift in the diagram
\diagramit{
 & E_1 \ar[d]^{\theta} \\
 \sfP \ar[r] \ar@{.>}[ur]& E_g
}
so the induced composite $\sfP \to E_1 \to E_B$ gives $B$ a $\sfP$-algebra structure in which both $f$ and $h$ become morphisms of $\sfP$-algebras.

In fact, we need only show $\theta$ is an acyclic fibration in $\sigmaec$, so for the remainder of the proof we work solely in $\sigmaec$.  Let the following pullback diagrams define their upper left corners
\diagramit{
E_2 \ar[d] \ar[r] & E_{B,C} \ar[d]  & E_3 \ar[r] \ar[d] & E_C \ar[d]   &  E_4 \ar[d] \ar[r] & E_2 \ar[d]
	& E_5 \ar[d] \ar[r] & E_3 \ar[d] \\
E_{A,B} \ar[r] & E_{A,C}               & E_2 \ar[r] & E_{B,C}                 &  E_A \ar[r] & E_{A,B}
	& E_4 \ar[r] & E_2
}
and notice $E_5$ is then canonically isomorphic to $E_g$, so we now consider
$E_1 \to E_5$.  Also notice there is an induced map $E_B \to E_2$, whose entries
are precisely the pushout-product maps
\diagramit{
B_{\ud}^{B_{\uc}} \ar[dr] \ar[ddr]_{(f_{\uc})^*} \ar[drr]^{(h_{\ud})_*} \\
&   E_2\binom{\ud}{\uc}   \ar[d] \ar[r] & C_{\ud}^{B_{\uc}} \ar[d]^{(f_{\uc})^*} \\
& B_{\ud}^{A_{\uc}} \ar[r]_{(h_{\ud})_*} & C_{\ud}^{A_{\uc}} .
}
The assumption about tensors preserving fibrations is needed to conclude $h_{\ud}$ is a fibration
in $\cate$, while $f_{\uc}$ is an acyclic cofibration in $\cate$, so the pushout-product axiom implies $E_B \to E_2$ is an (entrywise) acyclic fibration.

Using the universal property of pullbacks, we now have a commutative cube:
\diagramit{
E_1 \ar[rr] \ar[dr] \ar[dd] & & E_h \ar[rr] \ar[dr] \ar'[d][dd] & & E_C \ar'[d][dd] \ar@{=}[dr] & \\
& E_5 \ar[rr] \ar[dd] & & E_3 \ar[rr] \ar[dd] & & E_C \ar[dd] \\
E_f \ar'[r][rr] \ar[dr] \ar[dd] & & E_B \ar'[d][dd] \ar'[r][rr] \ar[dr] & & E_{B,C} \ar'[d][dd] \ar@{=}[dr] & \\
& E_4 \ar[dd] \ar[rr] & & E_2 \ar[dd] \ar[rr] & & E_{B,C} \ar[dd] \\
E_A \ar@{=}[dr] \ar'[r][rr] & & E_{A,B} \ar@{=}[dr] \ar'[r][rr] & & E_{A,C} \ar@{=}[dr] & \\
& E_A \ar[rr] & & E_{A,B} \ar[rr] & & E_{A,C}
}
Now the diagram
\diagramit{
E_h \ar[d] \ar[r] & E_3 \ar[d] \ar[r] & E_C \ar[d] \\
E_B \ar[r] & E_2 \ar[r] & E_{B,C}
}
in which the right square and the large rectangle are pullbacks, allows us to conclude the left square is a pullback, hence that $E_h \to E_3$ is an acyclic fibration as a base change of $E_B \to E_2$.  Similarly,
\diagramit{
E_f \ar[d] \ar[r] & E_B \ar[d] \\
E_4 \ar[d] \ar[r] & E_2 \ar[d] \\
E_A \ar[r] & E_{A,B}
}
with the rectangle and the bottom square pullbacks, allows us to conclude the top square is a pullback, and so the rectangle in
\diagramit{
E_1\ar[d] \ar[r] & E_h \ar[d] \\
E_f \ar[d] \ar[r] & E_B \ar[d] \\
E_4 \ar[r] & E_2
}
is also a pullback.  Since the composites agree by construction, the rectangle in
\diagramit{
E_1\ar[d] \ar[r] & E_h \ar[d] \\
E_5 \ar[d] \ar[r] & E_3 \ar[d] \\
E_4 \ar[r] & E_2
}
is then also a pullback, with the bottom square defined as one, so we conclude the top square is a pullback as well.  As a consequence, $E_1 \to E_5$ is now an acyclic fibration as desired, as a base change of the acyclic fibration $E_h \to E_3$.
\end{proof}

\begin{proof}[Proof of Theorem \ref{thm:algebralift}]
The forgetful functor $U$ preserves filtered colimits, so we can use Lifting Lemma \ref{basiclift}.  For a fibrant replacement functor, we notice $A \in \alg(\sfP)$ implies $\tilde A \in \alg(\tilde \sfP)$ by functoriality of the symmetric monoidal fibrant replacement functor $\tilde{(-)}$.  Here $\tilde \sfP$ indicates the fibrant replacement in $\prope$, obtained from $\sfP$ by applying the symmetric monoidal fibrant replacement functor to each component and map of $\sfP$.  Now $\tilde A$ equipped with the
$\sfP$-algebra structure associated to the composite $\sfP \to \tilde \sfP \to E_{\tilde A}$ serves as a fibrant replacement for $A$ in $\alg(\sfP)$.

For the required path object construction, let $A$ be an entrywise fibrant $\sfP$-algebra, and form
a strong path object in $\cate^\frakC$, by factoring the diagonal map
$A \stackrel{f}{\to} B \stackrel{h} \to A \times A$ with $f$ an acyclic cofibration and $h$ a fibration.  Since products are formed entrywise, $A \times A$ is also entrywise fibrant, so Lemma \ref{factorcompat} implies $B$ may be given a $\sfP$-algebra structure such that
both $f$ and $h$ are maps of $\sfP$-algebras.  By construction, $f$ is a weak equivalence (although we no longer claim it has the requisite lifting property of a cofibration of $\sfP$-algebras) and $h$ is a fibration, providing the requisite path object for fibrant objects in $\alg(\sfP)$.
\end{proof}

%%%%%%%%%%%%%%%%%%%%%%%%%%%%%%%%%%%%%%%%%%%%%%%%%
\section{Quillen pairs and Quillen equivalences}
\label{sec:Qpairs}
%%%%%%%%%%%%%%%%%%%%%%%%%%%%%%%%%%%%%%%%%%%%%%%%%

We would like to understand how various adjoint pairs interact with our model structures.
Throughout this section, we will simply assume the relevant ``projective" model structures exist, and focus on comparing them.  Keep in mind that the model structures may be constructed via methods technically different from those introduced in this article, but the results of this section will still apply.
We will also speak about algebras over PROPs, but similar definitions for algebras over $vPROPs$ or $hPROPs$ (Definition \ref{hpropdef}) would work equally well in what follows.

\begin{lemma}
	\label{adjoints}
	Suppose $(L,R)$ is an adjoint pair, $L:\cate \to \cate'$, where both functors
	are symmetric monoidal and $R$ preserves arbitrary
	coproducts and colimits indexed on connected
	groupoids.  Then they induce adjoint pairs at the level of $\frakC$-colored $\Sigma$-bimodules, $\frakC$-colored PROPs,
	vPROPs, hPROPs, and algebras.
\end{lemma}

\begin{proof}
	First, $(L,R)$ extend entrywise to bimodules, as a diagram category.	
	In the three PROP cases,
	we exploit the fact that $L$ and $R$ must preserve the relevant monoidal operations,
	which are constructed from the symmetric monoidal operation $\otimes$ in a natural way.
	Notice here we require the assumption on $R$, since the constructions
	$\Sigma_{\ub}$ (for $\boxv$) and
	$\Sigma_{\ud;\uc} \times \Sigma_{\ub;\ua}$ (for $\boxh$) involve coproducts and colimits
	indexed on connected groupoids.
	
	Finally, for the case of algebras, given $Y$ an $L\sfP$-algebra, there is an induced
	$\sfP$-algebra structure on $RY$ with structure maps
\diagramit{
\sfP\dc \otimes RY_{\uc} \ar[r]^-{\eta \otimes 1} & RL\sfP\dc \otimes RY_{\uc} \ar[r]
	& R\left(L\sfP\dc \otimes Y_{\uc}\right) \ar[d] \\
	&&RY_{\ud}
}
	with the last map given by the structure map of $Y$ as an $L\sfP$-algebra.
	Once again, we have extended $L$ and $R$ entrywise here in the product category
	underlying algebras, where they remain an adjoint pair.
 	Then an exercise involving
	the triangular identities for an adjunction and the assumption that both $L$ and $R$ are
	symmetric monoidal implies
\diagramit{
\sfP\dc \otimes X_{\uc} \ar[d] \ar[r] & X_{\ud} \ar[d] \\
\sfP\dc \otimes RY_{\uc} \ar[r] & RY_{\ud}
}	
	commutes in $\cate$ precisely when the corresponding diagram
\diagramit{
L\sfP\dc \otimes LX_{\uc} \ar[d] \ar[r] & LX_{\ud} \ar[d] \\
L\sfP\dc \otimes Y_{\uc} \ar[r] & Y_{\ud}
}
	commutes in $\cate'$, which completes the proof.	
\end{proof}

\begin{proposition}
	Suppose $(L,R)$ is a strong Quillen pair, $L:\cate \to \cate'$, where the
	component functors are symmetric monoidal and
	$R$ preserves arbitrary coproducts and colimits indexed on connected
	groupoids.  Then they induce
	strong Quillen pairs at the level of $\frakC$-colored $\Sigma$-bimodules, $\frakC$-colored PROPs, vPROPs, hPROPs, and
	algebras.  If, in addition, $(L,R)$ form a Quillen equivalence, then the same is true
	of each induced adjunction, whenever the entries of cofibrant objects are
	cofibrant in $\cate$.
\end{proposition}

\begin{proof}
	To verify that each adjunction forms a strong Quillen pair,
	it is simplest to observe that (acyclic) fibrations are defined in terms of
	entries in $\cate$ or $\cate'$, hence preserved by the entrywise
	right adjoints of Lemma \ref{adjoints}.
	For verifying the induced Quillen equivalence condition,
	it is key to know that the entries of a cofibrant
	object are cofibrant objects in $\cate$ in each case.
\end{proof}

\begin{remark}
	The ability to change our underlying model category up to Quillen equivalence is particularly gratifying since our technical assumptions have incidentally excluded important examples like (pointed) topological spaces.  Now we can conveniently say working instead with (pointed) simplicial sets is equally valid, as the cofibrant entries condition is vacuously satisfied over simplicial sets.  This cofibrant entries condition is also automatically satisfied for any category of $\frakC$-colored $\Sigma$-bimodules, as in
Proposition \ref{sigmabimods}.
\end{remark}

There are also adjoint pairs induced by changing colors; thus we assume we have a map
$\alpha:\frakC \to \frakC'$ of sets.

\begin{lemma}
	\label{changecolors}
	There are adjoint pairs given by changing colors, denoted $(\alpha_!,\alpha^*)$,
	at the level of colored $\Sigma$-bimodules, colored PROPs, vPROPs, hPROPs, and
	algebras.
\end{lemma}

\begin{proof}
	Here $\alpha^*$ is just a precomposition functor, where we again use $\alpha$ to denote
	the induced functor at the level of pairs of profiles $\catp_l(\frakC) \times \catp_r(\frakC) \to
	\catp_l(\frakC') \times \catp_r(\frakC')$.
	Hence, the usual left Kan extension formula for diagram categories
	gives the left adjoint at the level of colored
	$\Sigma$-bimodules.  Alternatively, we notice defining
\[\alpha_!\sfP\dcprime = \coprod_{\alpha(\uc)=\uc',\alpha(\ud)=\ud'} \sfP\dc
\]	
	has the requisite universal property to define this left adjoint.  With this formulation,
	it is not particularly difficult to verify that the operations from $\sfP$ induce those
	of $\alpha_!\sfP$ for the three PROP cases.
	
	For the case of algebras, we notice a similar coproduct formula will define the left
	adjoint to the precomposition functor we will again call $\alpha^*$ at the level of
	product categories underlying algebras.  Then we exploit the fact that $\otimes$ must
	distribute over coproducts by the existence of exponential objects, and the universal
	properties of coproducts to verify directly that the squares defining
	$\alpha_!X \to Y$ as a morphism of
	$\alpha_!\sfP$-algebras commute precisely when those defining
	$X \to \alpha^*Y$ as a morphism of $\sfP$-algebras commute.
\end{proof}

\begin{proposition}
	There is a strong Quillen pair $(\alpha_!,\alpha^*)$ at the level of
	colored $\Sigma$-bimodules, colored PROPs, vPROPs, hPROPs, and
	algebras associated to any $\alpha:\frakC \to \frakC'$.
\end{proposition}

\begin{proof}
	It again follows directly from the definitions that $\alpha^*$ preserves the class of (acyclic)
	fibrations in each case.
\end{proof}

\begin{remark}
	\label{colorQequiv}
	Notice in this change of colors adjunction, there is no claim of a Quillen equivalence.
	However, the following more general result does imply a Quillen equivalence of projective
	structures as expected in the case of a bijection of colors.
\end{remark}

\begin{corollary}
	\label{colorQequivs}
	If $\alpha:\frakC \to \frakC'$ is injective, there is a strong Quillen equivalence
	$(\alpha_!,\alpha^*)$ at the level of
	colored $\Sigma$-bimodules, colored PROPs, vPROPs, hPROPs, and
	algebras, only after using the modified
	projective structures (as in Corollary \ref{modproj} and Remark \ref{rmk:subset}), choosing only orbits consisting of colors in the image of $\alpha$.
\end{corollary}

\begin{proof}
	The point of using the modified structure on the $\frakC'$-colored structures is that
	weak equivalences are then created by the adjunction, while it remains a strong
	Quillen pair.  Hence, the pair is a Quillen equivalence precisely when the unit of
	the derived adjunction $X \to \alpha^*Q(\alpha_! X)$ is a weak equivalence for every cofibrant
	$\frakC$-colored object $X$, where $Q$ indicates a fibrant replacement in the $\frakC'$-colored
	structures \cite[1.3.16]{hovey}.  However, when $\alpha$ is injective, the formula for
	$\alpha_!$ given in the proof of Lemma \ref{changecolors} implies the unit of the adjunction
	$X \to \alpha^*\alpha_! X$ is the identity.  As noted above, $\alpha^*$ creates weak
	equivalences with this modified projective structure on the $\frakC'$-colored structures, so
	$\alpha_! X \to Q(\alpha_! X)$ a weak equivalence in this structure means
	precisely that $\alpha^*\alpha_! X \to  \alpha^*Q(\alpha_! X)$ is a weak equivalence, and so
	$X \to \alpha^*Q(\alpha_! X)$ is a weak equivalence.	
	\end{proof}

%%%%%%%%%%%%%%%%%%%%%%%%%%%%%%%%%%%%%%%%%%%%%%%
\section{Another description of colored PROPs}
\label{app:coloredPROPs}
%%%%%%%%%%%%%%%%%%%%%%%%%%%%%%%%%%%%%%%%%%%%%%%

In Definition \ref{def:prop}, we defined $\frakC$-colored PROPs as $\boxv$-monoidal $\boxh$-monoids.  In this section, we observe that the roles of $\boxv$ and $\boxh$ can be interchanged.  In other words, $\frakC$-colored PROPs can be equivalently defined as $\boxh$-monoidal $\boxv$-monoids.  We can first build $\frakC$-colored $\Sigma$-bimodules that have a horizontal composition $\otimes$ using a monoidal structure $\boxh$ on $\sigmaec$.  Then we build the vertical composition on top of the horizontal composition by considering a monoidal structure $\boxv$ on $\mon(\sigmaec,\boxh)$.  The monoids in this last monoidal category are exactly the $\frakC$-colored PROPs.

In more details, the functor $\boxh$ defined in \eqref{boxhsigmaobject} gives a monoidal product $\boxh \colon \sigmaec \times \sigmaec \to \sigmaec$ on $\sigmaec$, since its definition only involves the $\frakC$-colored $\Sigma$-bimodule structures on the two arguments.

\begin{definition}
\label{hpropdef}
Denote by $\hprope$ the category of monoids in the monoidal category $(\sigmaec, \boxh)$ (without unit), whose objects are called \textbf{hPROPs}.
\end{definition}

From the definitions of $\boxh$ and an associative map $\sfP \boxh \sfP \to \sfP$, an hPROP consists of exactly the data:
\begin{enumerate}
\item
An object $\sfP \in \sigmaec$.
\item
An associative horizontal composition $\otimes$ as in \eqref{horizontalcomposition}.
\end{enumerate}
Equivalently, the horizontal composition can be described in the level of $\cate$ as an associative operation \eqref{horizontalcomposition2} that is bi-equivariant \eqref{horizontalequivariance}.

To build the vertical compositions on top of the horizontal compositions, we consider the monoidal product $\boxv$ on $\sigmaec$ (Lemma \ref{boxvmonoidal}).  We need to upgrade $\boxv$ to a monoidal product on $\hprope$.  So suppose that $\sfP$ and $\sfQ$ are hPROPs.  Since $\sfP \boxv \sfQ \in \sigmaec$, to make $\sfP \boxv \sfQ$ into an hPROP, we need a horizontal composition
\begin{equation}
\label{hcompboxv}
(\sfP \boxv \sfQ)\binom{[\ud_1]}{[\uc_1]} \boxdot (\sfP \boxv \sfQ)\binom{[\ud_2]}{[\uc_2]} \xrightarrow{\otimes} (\sfP \boxv \sfQ)\binom{[\ud_1,\ud_2]}{[\uc_1,\uc_2]}
\end{equation}
that is induced by those on $\sfP$ and $\sfQ$.  Since $\boxv$ is defined as a coproduct, to define \eqref{hcompboxv} it suffices to define the operation
\begin{equation}
\label{hcompboxv2}
X \boxdot Y \xrightarrow{\otimes} Z
\end{equation}
in $\cate^{\Sigma_{\ud;\uc}}$, where
\[
X = \sfP\binom{[\ud_1]}{[\ub_1]} \otimes_{\Sigma_{\ub_1}} \sfQ\binom{[\ub_1]}{[\uc_1]},\,
Y = \sfP\binom{[\ud_2]}{[\ub_2]} \otimes_{\Sigma_{\ub_2}} \sfQ\binom{[\ub_2]}{[\uc_2]},\,
Z = \sfP\binom{[\ud]}{[\ub]} \otimes_{\Sigma_{\ub}} \sfQ\binom{[\ub]}{[\uc]},
\]
$\ud = (\ud_1,\ud_2)$, $\uc = (\uc_1,\uc_2)$, and $\ub = (\ub_1,\ub_2)$.  Recall that $\boxdot = K\iota$, where $K$ is a left Kan extension \eqref{leftKanext}.  Using the universal properties of left Kan extensions, to define \eqref{hcompboxv2} it suffices to define the operation
\begin{equation}
\label{hcompboxv3}
X \otimes Y \xrightarrow{\otimes} \cate^i(Z) \in \cate^{\Sigma_{\ud_1} \times \Sigma_{\ud_2} \times \Sigma_{\uc_1}^{op} \times \Sigma_{\uc_2}^{op}},
\end{equation}
where $\cate^i$ and $X \otimes Y$ are defined in \eqref{ei} and \eqref{XtensorY}, respectively.  To define the operation \eqref{hcompboxv3}, we go back to the definition of the functor $\otimes_{\Sigma_{\ub}}$ (\eqref{XtensoroverbY} and \eqref{XtensoroverbYmap}).

Consider an object $(\sigma_1\ud_1; \sigma_2\ud_2; \uc_1\tau_1^{-1}; \uc_2\tau_2^{-1}) \in \Sigma_{\ud_1} \times \Sigma_{\ud_2} \times \Sigma_{\uc_1}^{op} \times \Sigma_{\uc_2}^{op}$.  Then
\[
\begin{split}
(X \otimes Y)\left(\sigma_1\ud_1; \sigma_2\ud_2; \uc_1\tau_1^{-1}; \uc_2\tau_2^{-1}\right)
&= X\binom{\sigma_1\ud_1}{\uc_1\tau_1^{-1}} \otimes Y\binom{\sigma_2\ud_2}{\uc_2\tau_2^{-1}}\\
&= \left(\colim D_X\right) \otimes \left(\colim D_Y\right),
\end{split}
\]
where
\[
\begin{split}
D_X &= D_X\left(\sfP\binom{[\ud_1]}{[\ub_1]}, \sfQ\binom{[\ub_1]}{[\uc_1]}; \sigma_1\ud_1, \uc_1\tau_1^{-1}\right) \in \cate^{\Sigma_{\ub_1}},\\
D_Y &= D_Y\left(\sfP\binom{[\ud_2]}{[\ub_2]}, \sfQ\binom{[\ub_2]}{[\uc_2]}; \sigma_2\ud_2, \uc_2\tau_2^{-1}\right) \in \cate^{\Sigma_{\ub_2}}
\end{split}
\]
are the diagrams defined in \eqref{Dobj} and \eqref{tau''}.  Likewise, we have
\[
\cate^i(Z)(\sigma_1\ud_1; \sigma_2\ud_2; \uc_1\tau_1^{-1}; \uc_2\tau_2^{-1}) = \colim D_Z,
\]
where
\[
D_Z = D_Z\left(\sfP\binom{[\ud]}{[\ub]}, \sfQ\binom{[\ub]}{[\uc]}; (\sigma_1\ud_1, \sigma_2\ud_2), (\uc_1\tau_1^{-1}, \uc_2\tau_2^{-1})\right) \in \cate^{\Sigma_{\ub}}.
\]
Let $\mu_1\ub_1 \in \Sigma_{\ub_1}$ and $\mu_2\ub_2 \in \Sigma_{\ub_2}$.  Then we have a natural map
\begin{equation}
\varphi(\mu_1\ub_1,\mu_2\ub_2) \colon D_X(\mu_1\ub_1) \otimes D_Y(\mu_2\ub_2) \to \colim D_Z
\end{equation}
in $\cate$ that is defined as the composite:
\[
\SelectTips{cm}{10}
\xymatrix{
D_{X,Y} \ar@{=}[r] \ar[ddd]_-{\varphi(\mu_1\ub_1,\mu_2\ub_2)} & \left[\sfP\dbinom{\sigma_1\ud_1}{\ub_1\mu_1^{-1}} \otimes \sfQ\dbinom{\mu_1\ub_1}{\uc_1\tau_1^{-1}}\right] \otimes \left[\sfP\dbinom{\sigma_2\ud_2}{\ub_2\mu_2^{-1}} \otimes \sfQ\dbinom{\mu_2\ub_2}{\uc_2\tau_2^{-1}}\right] \ar[d]^-{\text{switch}}_-{\cong}\\
& \left[\sfP\dbinom{\sigma_1\ud_1}{\ub_1\mu_1^{-1}} \otimes \sfP\dbinom{\sigma_2\ud_2}{\ub_2\mu_2^{-1}}\right] \otimes \left[\sfQ\dbinom{\mu_1\ub_1}{\uc_1\tau_1^{-1}} \otimes \sfQ\dbinom{\mu_2\ub_2}{\uc_2\tau_2^{-1}}\right] \ar[d]^-{(\otimes, \otimes)}\\
& \sfP\dbinom{\sigma_1\ud_1, \sigma_2\ud_2}{\ub_1\mu_1^{-1}, \ub_2\mu_2^{-1}} \otimes \sfQ\dbinom{\mu_1\ub_1, \mu_2\ub_2}{\uc_1\tau_1^{-1}, \uc_2\tau_2^{-1}}\ar@{=}[d]\\
\colim D_Z & D_Z\left(\mu_1\ub_1, \mu_2\ub_2\right) \ar[l]_-{\eta}.
}
\]
Here $D_{X,Y} = D_X(\mu_1\ub_1) \otimes D_Y(\mu_2\ub_2)$, and the map $(\otimes, \otimes)$ has components the horizontal compositions (in the form ~\eqref{horizontalcomposition2}) in $\sfP$ and $\sfQ$, respectively.  The map $\eta$ is the natural map that comes with a colimit.

The maps $\varphi(\mu_1\ub_1,\mu_2\ub_2)$ are natural with respect to $\mu_1\ub_1$ and $\mu_2\ub_2$.  Now we fix $\mu_1\ub_1 \in \Sigma_{\ub_1}$ and let $\mu_2\ub_2$ vary through the category $\Sigma_{\ub_2}$.  Using the commutativity of $\otimes$ in $\cate$ with colimits, we obtain an induced map
\[
\varphi(\mu_1\ub_1) \colon D_X(\mu_1\ub_1) \otimes \left(\colim D_Y\right) \to \colim D_Z.
\]
Now, letting $\mu_1\ub_1$ vary through the category $\Sigma_{\ub_1}$, we obtain an induced map
\[
\varphi \colon \left(\colim D_X\right) \otimes \left(\colim D_Y\right) \to \colim D_Z.
\]
This map $\varphi$ is the required map \eqref{hcompboxv3} when applied to a typical object
\[
(\sigma_1\ud_1; \sigma_2\ud_2; \uc_1\tau_1^{-1}; \uc_2\tau_2^{-1}) \in \Sigma_{\ud_1} \times \Sigma_{\ud_2} \times \Sigma_{\uc_1}^{op} \times \Sigma_{\uc_2}^{op}.
\]

The naturality of the construction $\varphi$ with respect to maps in $\Sigma_{\ud_1} \times \Sigma_{\ud_2} \times \Sigma_{\uc_1}^{op} \times \Sigma_{\uc_2}^{op}$ is easy to check.  So we have constructed the operation \eqref{hcompboxv3}, and hence the operation \eqref{hcompboxv}.  The associativity of \eqref{hcompboxv} follows from that of \eqref{hcompboxv3}, which in turn follows from the naturality of the construction $\varphi$ and the associativity of the horizontal compositions in $\sfP$ and $\sfQ$.  Thus, we have shown that $\boxv$ (Lemma \ref{boxvmonoidal}) extends to a monoidal product on the category $\hprope$.

As before, we consider the category of monoids $\mon\left(\hprope, \boxv\right)$ in the monoidal category $\left(\hprope, \boxv\right)$.  Unwrapping the meaning of a monoid, it is straightforward to check that $\mon\left(\hprope, \boxv\right)$ is canonically isomorphic to the category $\prope$ (Proposition \ref{coloredprop}).  In this case, the interchange rule says that the monoid map $\circ \, \colon \sfP \boxv \sfP \to \sfP$ is a map of hPROPs.  This is, in fact, equivalent to the original interchange rule \eqref{interchangerule} due to its symmetry.  In summary, $\frakC$-colored PROPs can be equivalently described as:
\begin{enumerate}
\item
$\boxv$-monoidal $\boxh$-monoids, as in Proposition \ref{coloredprop}, or
\item
$\boxh$-monoidal $\boxv$-monoids.
\end{enumerate}
Symbolically, we have
\[
\begin{split}
\prope &= \mon\left(\vprope, \boxh\right) = \mon\left(\mon\left(\sigmaec, \boxv\right), \boxh\right) \\
& \cong \mon\left(\hprope, \boxv\right) = \mon\left(\mon\left(\sigmaec, \boxh\right), \boxv\right).
\end{split}
\]
In the first description, we consider a $\frakC$-colored PROP $\sfP$ as a $\frakC$-colored $\Sigma$-bimodule with a vertical composition $\circ$ (i.e., a vPROP) together with a horizontal composition $\otimes$ that is built on top of $\circ$.  In the second description, we consider $\sfP$ as a $\frakC$-colored $\Sigma$-bimodule with a horizontal composition $\otimes$ (i.e., an hPROP) together with a vertical composition $\circ$ that is built on top of $\otimes$.

%%%%%%%%%%%%%%%%%%%%%%%%%%%%%%%%%%%%%%%%%%%%
\section{Colored operads and colored PROPs}
\label{app:operads}
%%%%%%%%%%%%%%%%%%%%%%%%%%%%%%%%%%%%%%%%%%%%

In this section, we construct the (Quillen) adjunction between colored operads and colored PROPs (Proposition \ref{operadpropadj} and Corollary \ref{cor:operad}), generalizing what is often used in the $1$-colored chain case \cite[Example 60]{markl06}.  In one direction, it associates to a colored operad the free colored PROP generated by it.  This free functor involves a functor $\boxdot$ (Lemma \ref{boxdotassociative}) that is the main ingredient of the monoidal product $\boxh$ \eqref{boxhsigmaobject}.  In the other direction, it associates to a colored PROP its underlying colored operad.
In fact, we will show that a modified projective structure on $\prope$ can be chosen to turn this
adjunction into a strong Quillen equivalence (Proposition \ref{prop:operadQequiv}).  As a consequence, the homotopy theory of $\prope$ in its projective structure, having fewer weak equivalences than this modified projective structure, is a refinement of the projective homotopy theory for colored operads.  We will also see that, for a colored operad $\sfO$, its free colored PROP $\sfO_{prop}$ has an equivalent category of algebras (Corollary \ref{Opropalgebra}).  In other words, going from a colored operad $\sfO$ to the colored PROP $\sfO_{prop}$ does not change the algebras.

%%%%%%%%%%%%%%%%%%%%%%%%%%%%%%%%%%%%%%%%%%%%%%
\subsection{Colored operads and colored PROPs}

We refer the reader to \cite{may72,may97} for the well-known definitions of operads, endomorphism operads, and algebras over an operad. The definitions in the colored case can be found in, e.g., \cite[Section 2]{markl04}.  The category of $\frakC$-colored operads over $\cate$ is denoted by $\operade$.

\begin{proposition}
\label{operadpropadj}
There is a pair of adjoint functors
\begin{equation}
\label{operadprop}
(-)_{prop} \colon \operade \rightleftarrows \prope \colon U
\end{equation}
between the categories of $\frakC$-colored PROPs and $\frakC$-colored operads, where the right adjoint $U$ is the forgetful functor.
\end{proposition}

\begin{proof}
First we construct the forgetful functor $U$.  Suppose that $d, c_i, b^i_j \in \frakC$ are colors, where $1 \leq i \leq n$ and, for each $i$, $1 \leq j \leq k_i$.  Write $\uc = (c_1, \ldots , c_n)$, $\ub^i = (b^i_1, \ldots , b^i_{k_i})$, and $\ub = (\ub^1, \ldots , \ub^n)$.  If $\sfP$ is a $\frakC$-colored PROP, then the components in the $\frakC$-colored operad $U\sfP$ are
\begin{equation}
\label{Uprop}
(U\sfP)\binom{d}{c_1, \ldots , c_n} = \sfP\binom{d}{c_1, \ldots , c_n} = \sfP\binom{d}{\uc}.
\end{equation}
The structure map $\rho$ of the $\frakC$-colored operad $U\sfP$ is the composition
\begin{equation}
\label{Upropstructuremap}
\SelectTips{cm}{10}
\xymatrix{
\sfP\dbinom{d}{\uc} \otimes \sfP\dbinom{c_1}{\ub^1} \otimes \cdots \otimes \sfP\dbinom{c_n}{\ub^n} \ar[dr]^-{\rho} \ar[d]_-{Id \otimes \text{(horizontal)}} & \\
\sfP\dbinom{d}{\uc} \otimes \sfP\dbinom{\uc}{\ub} \ar[r]^-{\circ} & \sfP\dbinom{d}{\ub}.
}
\end{equation}
The associativity of the horizontal and the vertical compositions in $\sfP$ together with the interchange rule \eqref{interchangerule} imply that $\rho$ is associative.  The equivariance of $\rho$ follows from those of $\otimes$ and $\circ$.

Now we construct the unique colored PROP $\sfO_{prop}$ generated by a colored operad $\sfO$.  Let $\sfO$ be a $\frakC$-colored operad with components
\[
\sfO\binom{d}{c_1, \ldots , c_n} = \sfO\binom{d}{\uc}
\]
for $d, c_i \in \frakC$.  First we define the underlying $\frakC$-colored $\Sigma$-bimodule of $\sfO_{prop}$.  We have to specify the diagrams
\[
\sfO_{prop}\binom{[\ud]}{[\uc]} \in \cate^{\Sigma_{\ud;\uc}} = \cate^{\Sigma_{\ud} \times \Sigma_{\uc}^{op}},
\]
where $\ud = (d_1, \ldots , d_m)$ and $\uc = (c_1, \ldots , c_n)$ are $\frakC$-profiles.  To each partition $r_1 + \cdots + r_m = n$ of $n$ with each $r_i \geq 1$, we can associate to the $\frakC$-colored operad $\sfO$ the diagrams
\[
\sfO\binom{[d_i]}{[\uc_i]} \in \cate^{\Sigma_{d_i; \uc_i}} = \cate^{\Sigma_{d_i} \times \Sigma_{\uc_i}^{op}} = \cate^{\{*\} \times \Sigma_{\uc_i}^{op}}
\]
for $1 \leq i \leq m$, where
\[
\uc_i = \left(c_{r_1 + \cdots + r_{i-1} + 1}, \ldots , c_{r_1 + \cdots + r_i}\right).
\]
Recall the associative functor
\[
\boxdot = K \iota \colon \cate^{\Sigma_{\ud_1;\uc_1}} \times \cate^{\Sigma_{\ud_2;\uc_2}} \to \cate^{\Sigma_{(\ud_1,\ud_2);(\uc_1,\uc_2)}}
\]
from Lemma \ref{boxdotassociative}, where $\iota$ is an inclusion functor and $K$ is a left Kan extension.  Using the associativity of $\boxdot$, we define the object
\begin{equation}
\label{Oprop}
\sfO_{prop}\binom{[\ud]}{[\uc]} = \coprod_{r_1 + \cdots + r_m = n} \sfO\binom{[d_1]}{[\uc_1]} \boxdot \cdots \boxdot \sfO\binom{[d_m]}{[\uc_m]}
\end{equation}
in $\cate^{\Sigma_{\ud;\uc}}$, where the coproduct is taken over all the partitions $r_1 + \cdots + r_m = n$ with each $r_i \geq 1$.  This defines $\sfO_{prop}$ as an object in $\sigmaec$.

The horizontal composition in $\sfO_{prop}$ is given by concatenation of $\boxdot$ products and inclusion of summands.  Using the universal properties of left Kan extensions, the vertical composition in $\sfO_{prop}$ is uniquely determined by the operad composition in $\sfO$.  It is straightforward to check that $(-)_{prop}$ is left adjoint to the forgetful functor $U$ from $\frakC$-colored PROPs to $\frakC$-colored operads.
\end{proof}

Note that the left adjoint $(-)_{prop}$ is an embedding.  In fact, for a $\frakC$-color operad $\sfO$, it follows from the definitions of $(-)_{prop}$ and $U$ that $\sfO = U (\sfO_{prop})$.

Under the hypotheses of Theorem \ref{thm:coloredpropmodel}, the category $\operade$ of $\frakC$-colored operads in $\cate$ also has a cofibrantly generated model category structure in which the fibrations and weak equivalences are defined entrywise in $\cate$ (see \cite[Theorem 3.2]{bm03} and \cite[Theorem 2.1 and Example 1.5.7]{bm07}).

\begin{corollary}
\label{cor:operad}
Under the hypotheses of Theorem \ref{thm:coloredpropmodel}, the adjoint pair $((-)_{prop},U)$ in Proposition \ref{operadpropadj} is a Quillen pair.
\end{corollary}

\begin{proof}
From the definition \eqref{Uprop}, $U$ preserves (acyclic) fibrations.
\end{proof}

\begin{proposition}
\label{prop:operadQequiv}
Under the hypotheses of Theorem \ref{thm:coloredpropmodel}, the adjoint pair $((-)_{prop},U)$ in Proposition \ref{operadpropadj} is a Quillen equivalence, if $\prope$ is given the modified projective
structure (as discussed in Remark \ref{rmk:subset}) determined by components with a single object in the target.
\end{proposition}

\begin{proof}
	This proof is quite similar to that for Corollary \ref{colorQequivs}.  Again, the choice of modified projective
	structure implies the right adjoint creates weak equivalences, with the unit of the adjunction
	an isomorphism.  As before, this implies the right adjoint preserves the weak equivalence
	built into the unit of the derived adjunction, so \cite[1.3.16]{hovey} suffices.
\end{proof}

We now observe that passing from a colored operad $\sfO$ to the colored PROP $\sfO_{prop}$ does not alter the category of algebras.

\begin{corollary}
\label{Opropalgebra}
Let $\sfO$ be a $\frakC$-colored operad.  Then there are functors
\[
\Phi \colon \alg(\sfO) \rightleftarrows \alg(\sfO_{prop}) \colon \Psi
\]
that give an equivalence between the categories $\alg(\sfO)$ of $\sfO$-algebras and $\alg(\sfO_{prop})$ of $\sfO_{prop}$-algebras.
\end{corollary}

\begin{proof}
First observe that in each of the two categories, an algebra has an underlying $\frakC$-graded object $\{A_c\}$.  Given an $\sfO$-algebra $A$, the formula \eqref{Oprop} for $\sfO_{prop}$ together with the universal properties of left Kan extensions extend $A$ to an $\sfO_{prop}$-algebra.  This is the functor $\Phi$.

Conversely, an $\sfO_{prop}$-algebra is a map $\lambda \colon \sfO_{prop} \to E_X$ of $\frakC$-colored PROPs, where $E_X$ is the $\frakC$-colored endomorphism PROP of a $\frakC$-graded object $X = \{X_c\}$.  Using the free-forgetful adjunction from Proposition \ref{operadpropadj}, this $\sfO_{prop}$-algebra is equivalent to a map $\lambda' \colon \sfO \to U(E_X)$ of $\frakC$-colored operads.  From the definition (\eqref{Uprop} and \eqref{Upropstructuremap}) of the forgetful functor $U$, one observes that $U(E_X)$ is the $\frakC$-colored endomorphism operad of $X$.  Therefore, the map $\lambda'$ is actually giving an $\sfO$-algebra structure on $X$.  This is the functor $\Psi$.  One can check that the functor $\Phi$ and $\Psi$ give an equivalence of categories.
\end{proof}

%%==============%%
%%              %%
%%  References  %%
%%              %%
%%==============%%

\end{document}